\documentclass[a4paper,12pt]{article}

\usepackage[utf8]{inputenc}
\usepackage[english]{babel}

\usepackage{latexsym}
\usepackage{amsfonts}
\usepackage{amsmath}
\usepackage{amssymb}
\usepackage{graphicx}
\usepackage{color}
\usepackage{enumerate}

\usepackage{amsthm}

\newcommand{\HS}{\mathbb{R}^3_{\scalebox{1.0}[1.0]{\,-}}}
\newcommand{\CHS}{\overline{\mathbb{R}^3_{\scalebox{0.5}[1.0]{\( - \)}}}}

\newcommand{\Bm}{B_{\scalebox{0.5}[1.0]{\( - \)}}}

\newcommand{\R}{\mathbb{R}}

\newcommand{\C}{\mathbb{C}}

\newcommand{\nuoli}{\rightarrow}
\newcommand{\Lc}{L_{comp}}
\newcommand{\Hc}{H_{comp}}
\newcommand{\Wc}{W_{comp}}
\newcommand{\Ll}{L_{loc}}
\newcommand{\Hl}{H_{loc}}

\newcommand{\Curl}{\nabla \times}

\newcommand{\ov}[1]{\overline{#1}}

\newcommand{\ksi}{\xi}
\newcommand{\dotc}{\cdot}

\renewcommand\Re{\operatorname{Re}}
\renewcommand\Im{\operatorname{Im}}
\newcommand\supp{\operatorname{supp}} 
 
\newcommand\spec{\operatorname{Spec}} 
 
\newcommand\Dom{\operatorname{Dom}} 
\newcommand\p{\operatorname{\partial}} 

\newcommand{\Ld}{L_{A,q}}

\newtheorem{thm}{Theorem}[section]
\newtheorem{cor}[thm]{Corollary}
\newtheorem{lem}[thm]{Lemma}
\newtheorem{prop}[thm]{Proposition}

\newtheorem{defn}[thm]{Definition}
\newtheorem{rem}[thm]{Remark}

\theoremstyle{definition}

\numberwithin{equation}{section}

\begin{document}
\nocite{*}
\title{An Inverse Boundary Value Problem for the  Magnetic Schr\"odinger 
Operator on a Half Space}
\author{Valter Pohjola}
\maketitle
\vfill
\begin{flushright}
\noindent
University of Helsinki \\
Department of Mathematics and Statistics \\
Licentiate Thesis \\
Advisors: Katya Krupchyk, Lassi Päivärinta
\end{flushright}
\thispagestyle{empty}

\newpage
\textbf{Abstract:} This licentiate thesis is concerned with an inverse boundary value
problem for the magnetic Schr\"odinger equation in a half space, for potentials
$A\in \Wc^{1,\infty}(\CHS,\R^3)$ and $q \in \Lc^{\infty}(\CHS,\C)$. 
We prove that $q$ and the curl of $A$ are uniquely determined by the knowledge
of the Dirichlet-to-Neumann map on parts of the boundary of the half space.
The existence and uniqueness of the corresponding direct problem are also
considered.
\thispagestyle{empty}

\newpage
\tableofcontents
\thispagestyle{empty}

\newpage
\section{Introduction}

The main purpose of this thesis is to investigate an inverse problem for the 
magnetic Schr\"odinger operator in the half space geometry. 
The  magnetic Schrödinger operator $\Ld$ is defined by
\begin{eqnarray}
\Ld := \sum_{j=1}^3 (-i\partial_j + A_j)^2 + q(x). \label{eq:Ldexp}
\end{eqnarray}
We fix the half space by considering $\HS:=\{ x \in \R^3 \;|\; x_3 < 0 \}$. 
We shall moreover assume that 
\begin{align}
A \in \Wc^{1,\infty}(\CHS,\R^3),   \text{ and }
q \in \Lc^{\infty}(\CHS,\C),   \quad \Im q\le 0.
\end{align}
Here 
\[
\Wc^{1,\infty}(\CHS,\R^3):=\{A|_{\CHS}  \; \big|\, A\in W^{1,\infty}(\R^3, \R^3), \supp(A)\subset 
\R^3 \text{ compact}\}
\]
and similarly, we define
\[
\Lc^{\infty}(\CHS,\C):=\{q\in L^{\infty}(\CHS,\C)\;\big|\,\supp(q)\subset 
\CHS \text{ compact}\}.
\]

The direct problem, from which the inverse problem stems, 
is the Dirichlet problem
\begin{equation}
 \left.\begin{aligned}
        (\Ld - k^2) u & = 0 \quad \text{ in $\HS$,}\\
        u |_{\partial \HS} &= f,
       \end{aligned}
 \right.
 \label{eq:BVP}
\end{equation}
where $k>0$ is fixed and  $f \in \Hc^{3/2}(\p\HS)$. 
Furthermore,  we will also require that the solution $u$ should satisfy a 
boundary
condition at infinity, which will be the
\textit{Sommerfeld radiation condition}
\begin{equation}
\lim_{|x| \nuoli \infty} |x| \bigg( \frac{\partial u(x)}{ \partial |x|}
-iku(x)\bigg) = 0.
\label{eq:SRC} 
\end{equation}
Solutions satisfying this condition are called \textit{outgoing}
or \textit{radiating} solutions. We will also occasionally use the term
\textit{incoming} solution. This refers to a solution of \eqref{eq:BVP} that 
satisfies 
\eqref{eq:SRC}, when the factor $-ik$ is replaced by $ik$.

The following result, which will be established in the first part of this work, 
gives the solvability of the direct problem \eqref{eq:BVP}, 
\eqref{eq:SRC}. 

\begin{thm}
\label{thm_1_direct}
Let $A\in \Wc^{1,\infty}(\CHS,\R^3)$ and $q \in \Lc^{\infty}(\CHS,\C)$ be such 
that  $\Im q\le 0$. Then for any $f \in \Hc^{3/2}(\p\HS)$, the Dirichlet 
problem \eqref{eq:BVP}, 
\eqref{eq:SRC} has a unique solution 
$u\in \Hl^2(\CHS)$. 
\end{thm}

Here
\[
H^2_{loc}(\overline{\HS}):=\{u|_{\overline{\HS}} \; \big|\, u\in H^2_{loc}(\R^3)\}.
\]

Theorem \ref{thm_1_direct} permits us to define the so called \textit{Dirichlet 
to Neumann map} $\Lambda_{A,q}$,
(DN-map for short),
$\Lambda_{A,q}:\Hc^{3/2}(\p \HS) \nuoli H^{1/2}_{loc}(\p \HS)$ as
\begin{align*}
f\mapsto(\partial_n + i A\cdot n) u |_{\p \HS},
\end{align*}
where $u$ is the solution of the Dirichlet problem \eqref{eq:BVP}, 
\eqref{eq:SRC} and $f$ is the boundary condition. 
Here $n=(0,0,1)$ is the unit outer normal to the boundary $\p \HS$.

The inverse problem is then to determine if the DN-map uniquely determines the
potentials $A$ and $q$ in $\HS$. It turns out that the DN-map does not
in general uniquely determine $A$. This is due to the  gauge invariance of
the DN-map, which was first noticed by \cite{S1}.
It follows from the identities  
\begin{equation}
\label{eq_gauge}
e^{-i\psi}L_{A,q}e^{i\psi}=L_{ A+\nabla\psi,  q},\quad 
e^{-i\psi}\Lambda_{ A, q} e^{i\psi}=\Lambda_{A+\nabla \psi,q},
\end{equation}
that  $\Lambda_{A, q} =\Lambda_{A+\nabla \psi, q}$ when $\psi\in 
C^{1,1}(\CHS,\R)$ compactly supported is such that 
$\psi|_{\p\HS}=0$ (see Lemma \ref{gauge_inv}).  
This shows that $\Lambda_{A,q}$ cannot uniquely determine $A$.
The DN-map does however carry enough information to determine
$\nabla\times A$, which is the magnetic field in the context of electrodynamics. 

We shall use the notation  $A_j=(A_{j,1},A_{j,2},A_{j,3})$ for the component
functions, 
when considering a pair of magnetic potentials $A_j$, $j=1,2$. 
We now state the main result of this work, which generalizes the 
corresponding results of \cite{LCU1}, obtained in the case of the  
Schr\"odinger operator without a magnetic potential. 
\begin{thm}
\label{thm_2_inverse}
Let $A_j\in \Wc^{1,\infty}(\CHS,\R^3)$ and $q_j \in \Lc^{\infty}(\CHS,\C)$ be 
such that  $\Im q_j\le 0$, $j=1,2$. 
Denote by $B$ an open ball in $\R^3$, containing the supports of $A_j$, and $q_j$, $j=1,2$. 
Let $\Gamma_1,\Gamma_2\subset\p \R^3_-$ be open sets such that
\begin{align}
(\p \HS \setminus \ov{B}) \cap \Gamma_j \neq \emptyset, \quad j=1,2.
\label{eq_gamma}
\end{align}
Then if 
\begin{equation}
\label{eq_data_inv}
\Lambda_{A_1,q_1}(f)|_{\Gamma_1}=\Lambda_{A_2,q_2}(f)|_{\Gamma_1},
\end{equation}
for all $f\in \Hc^{3/2} (\p \HS)$, $\supp(f)\subset 
\Gamma_2$, then 
\[
\nabla\times A_1=\nabla \times A_2\quad\text{and}\quad q_1=q_2\quad 
\text{in}\quad \HS.
\]
\end{thm}

We would like to emphasize that in Theorem \ref{thm_2_inverse}, the set 
$\Gamma_1$, where measurements are performed, and the set $\Gamma_2$, where the 
data is supported,  can be taken arbitrarily small, provided that 
\eqref{eq_gamma} holds.  The result of Theorem \ref{thm_2_inverse}  pertains 
therefore to inverse problems with partial data. Such problems are important 
from the point of view of applications, since in practice, performing 
measurements on the entire boundary could be impossible, due to limitations in 
resources or obstructions from obstacles.  

The first uniqueness result, 
in the context of inverse boundary value problems for the magnetic 
Schr\"o\-dinger operator on a bounded domain, was 
obtained by Sun in \cite{S1}, under a smallness condition on $A$. 
Nakamura, Sun and Uhlmann proved the uniqueness without any smallness condition 
in \cite{NSU1}, assuming that $A \in C^2$. Tolmasky extended this result to
$C^1$ magnetic potentials, in \cite{Tol1} and Panchenko to some less regular but small
magnetic potentials in \cite{Pan1}.
Salo proved uniqueness for Dini continuous magnetic potentials in \cite{MS1}. 
The most recent result is given by Krupchyk and Uhlmann in \cite{Kru1}, where
uniqueness is proved for $L^\infty$ magnetic potentials.
In all of these works, the inverse boundary value 
problem with full data was considered. 

In \cite{ER1}, Eskin and Ralston obtained a uniqueness result for the closely 
related inverse scattering problem, assuming the exponential decay of the 
potentials.
The partial data problem in the magnetic case was considered by Dos Santos 
Ferreira, Kenig, Sjöstrand and Uhlmann in \cite{DosSantos1} and by Chung 
in \cite{Chung1}.

The inverse problem for the half space geometry, without a magnetic potential 
was examined by Cheney and Isaacson in \cite{CI1}. The uniqueness for this 
problem in the case of compactly supported electric potentials was proved
by Lassas, Cheney and Uhlmann in \cite{LCU1}, assuming that the supports do not 
come close to the boundary of the half space. The result of  Theorem 
\ref{thm_2_inverse} is therefore already a generalization of the work  
\cite{LCU1}, even in the absence of magnetic potentials. Li and Uhlmann  proved 
uniqueness for the closely related infinite slab geometry with $A=0$, in 
\cite{LU1}. Krupchyk, Lassas and Uhlmann did this for the magnetic case in
\cite{KLU1}. In both of these works, the reflection argument of Isakov 
\cite{I1} played an important role. The uniqueness problem for the magnetic potentials
in the slab and half space geometries 
has also been studied in a recent paper by Li \cite{Li1}. The half space results
in \cite{Li1} differ from the ones given in this work, by concerning the more general 
matrix valued Schrödinger equation and by assuming $C^6$ regularity on the  
magnetic potential.

The half space is perhaps the simplest example of an unbounded region with an 
unbounded boundary. It is of special interest in many applications, such as 
geophysics, ocean acoustics, and optical tomography, 
since it provides a simple model for semi infinite geometries. We would like to mention that 
the magnetic Schrödinger equation is closely related to the diffusion 
approximation of the photon transport equation, used in optical
tomography \cite{SA1}. The half space geometry is also of interest in optical 
tomography, since in practice, the source-detector pairs are often located on 
the same interface \cite{KL1}. 

The thesis  is divided into two main parts.  Section 2  gives a detailed 
account of the solvability of the direct problem and provides a proof of 
Theorem \ref{thm_1_direct}.  
Subsection \ref{sec:GF} develops some basic tools from
scattering theory. In Subsection \ref{sec:dirWS} we prove the existence and
uniqueness for the direct scattering problem in all of $\R^3$, using the 
Lax--Phillips method. In Subsection \ref{sec:dirHS} we extend this discussion to
the half space case, using a reflection argument.

Section 3 deals with the inverse problem and contains the proof of Theorem 
\ref{thm_2_inverse}.  In Subsection 3.1 a central integral identity is derived. 
 Subsection 3.2 contains the construction of complex geometric optics  
solutions   for magnetic Schr\"odinger operators with Lipschitz continuous 
potentials. The proof of Theorem \ref{thm_2_inverse} is concluded in 
Subsections 3.3 and 3.4.

\subsection{Acknowledgements}

I wish to thank my advisor Katya Krupchyk for suggesting the topic of this
thesis and for all the help while writing it. I also like to thank Mikko Salo
for providing me with his notes on Carleman estimates for the Magnetic Schr\"odinger
operator.

\newpage
\section{The direct problem}

The purpose of this section is to investigate the well-posedness of the 
boundary value problem \eqref{eq:BVP}, \eqref{eq:SRC} for
the magnetic Schr\"odinger operator in the half space, and to establish Theorem 
\ref{thm_1_direct}. 

 We prove existence 
and uniqueness results for this problem, but we will not concern ourselves with
showing that solutions depend continuously on boundary data. This is enough 
to guarantee that the
questions of the uniqueness of the inverse problem are sensible.

\subsection{The free space outgoing Green function}
\label{sec:GF}

The aim of this subsection is to introduce the outgoing Green function for the 
Helmholtz
equation and develop some basic notions used in scattering theory that we need
to attack the direct problem. For a more
in depth exposition,  see \cite{Colton1}. Having constructed the outgoing Green 
function, we use it to
investigate the corresponding resolvent operator and the asymptotics of
solutions.

A function $G$ is generally speaking a Green function for the Helmholtz
equation if it solves the problem 
\begin{eqnarray}
(-\Delta_x - k^2)G(x,y) = \delta(x-y), \quad y\in \Omega, 
\label{eq:GreensF}
\end{eqnarray}
in some region $\Omega$ and satisfies some specific boundary condition on
$\partial \Omega$. We are interested in the case $\Omega = \R^3$ with 
the boundary condition being the Sommerfeld radiation condition \eqref{eq:SRC}. 
 
In the next proposition we construct a specific Green function called the
\textit{outgoing free space Green function} denoted by $G_0$, which 
satisfies \eqref{eq:GreensF} and the Sommerfeld condition.

\begin{prop} \label{propGreen}
Let $k>0$. The function
\begin{equation}
G_0(x,y) = \frac{e^{i k |x-y|}}{4 \pi |x-y|}, \label{eq:FSGF}
\end{equation}
is a free space outgoing Green function.
\end{prop}

\begin{proof} We start by assuming $y=0$ in
(\ref{eq:GreensF}) and thinking of $G_0$ as a function of $x$ only.
Instead of considering the operator  $-\Delta -k^2$ directly, we first look 
for outgoing
Green's functions $G_{\epsilon}$ for the operators $-\Delta -k^2 -
i\epsilon$, where $\epsilon > 0$. By inserting $i \epsilon$ into 
\eqref{eq:GreensF} and taking the
Fourier transform we get that
\begin{align*}
\widehat{G}_{\epsilon} = \frac{1}{\xi^2 - k^2 - i\epsilon}.
\end{align*}
Since this is a locally integrable function, decaying at infinity,  we may take 
the inverse Fourier transform.
Let us therefore proceed by calculating the inverse Fourier transform,
\begin{align}
\mathcal{F}^{-1} \bigg( \frac{1}{ \xi^2 - k^2 - i \epsilon } \bigg)(x) &=
\frac{1}{(2 \pi)^3} \int_{\R^3} \frac{e^{i \xi \cdot x}} { \xi^2 -k^2 -
i \epsilon } d \xi \label{eq:Fep} \\
&= \frac{1}{(2 \pi)^3} \int_{0}^{2 \pi} \int_{0}^{\pi} \int_{0}^{\infty} 
\frac{e^{i
\xi(r,\theta,\phi) \cdot x}}
{r^2 -k^2 - i\epsilon } r^2 \sin \theta  d\phi d\theta dr. \nonumber
\end{align}
The Fourier transform of a spherically symmetric function is spherically
symmetric. We need therefore to calculate this integral only for e.g. 
$x=(0,0,R)$, $R=|x|$.
By doing this and abbreviating $\alpha = k^2 + i \epsilon $   we get 
\begin{eqnarray*}
\int_{0}^{\pi} \int_{0}^{\infty} \frac{e^{i rR \cos \theta}}
{ r^2 - \alpha} r^2 \sin \theta  d\theta dr &=& 
\frac{-1}{iR} \int_{0}^{\infty} \bigg[_{0}^{\pi} 
\frac{e^{i rR cos\theta}}{ r^2 - \alpha } r dr \\
&=& \frac{-1}{iR} \int_{0}^{\infty} 
\frac{e^{-i rR} - e^{i rR}}{ r^2 - \alpha} r dr \\
&=& \frac{1}{iR} \int_{-\infty}^{\infty} 
\frac{e^{i rR}}{ r^2 - \alpha} r dr. 
\end{eqnarray*}
We can evaluate this last integral with the method of residues, if we
temporarily replace
$r$ with a complex variable $z$. We choose as a
contour, an origin centered half circle $C_s$ that is in the upper half plane 
and
goes along the real axis.
The parameter $s$ stands for radius of the circle.  
By choosing the branch of the complex square root in the upper half plane  and 
a large enough $s$,
we achieve that  the pole $\sqrt{\alpha}$ will lie inside the contour. 
A residue integration gives us then
\begin{align*}
\int_{-\infty}^{\infty} \frac{e^{i rR}}{ r^2 - \alpha } r dr  
&= \lim_{s\nuoli \infty} \oint_{C_s} 
\frac{e^{i zR}}{ (z - \sqrt{\alpha})(z  + \sqrt{\alpha}) } z dz \\ 
&= 2\pi i \operatorname{Res} 
\bigg(\frac{e^{i zR}z}{ (z - \sqrt{\alpha})(z  +\sqrt{\alpha})}, \sqrt{\alpha} 
\bigg) \\ 
&= \pi i e^{iR\sqrt{\alpha}}.
\end{align*}
By going back to \eqref{eq:Fep} we see that
\begin{eqnarray*}
\mathcal{F}^{-1} \bigg( \frac{1}{ \xi^2 - k^2 - i \epsilon } \bigg)(x) =
\frac{e^{i\sqrt{k^2+i\epsilon}R}}{4 \pi R}, \quad R=|x|.
\end{eqnarray*}
In other words, $G_\epsilon(x) = e^{i\sqrt{k^2+i\epsilon}|x|}/(4 \pi |x|)$.

We will need to use the theory of distributions to extend the above argument to
the case $\epsilon = 0$.  We define the tempered
distribution $(\xi^2 -k^2 -i0)^{-1}$, by taking the limit
\begin{align}
\big \langle (\xi^2 -k^2 -i0)^{-1}, \varphi \big  \rangle  
= \lim_{\epsilon \to 0+} \big  \langle (\xi^2 - k^2 -i \epsilon)^{-1}
,\varphi \big \rangle 
\label{eq:regdef}
\end{align}
for $\varphi \in \mathcal{S}(\R^3)$, where the distribution 
pairing on the right is given by an integral.
One can check that this definition makes sense, using similar arguments as for 
the principal value distribution.
We define 
\begin{align*}
G_0 := \mathcal{F}^{-1} \big( (\xi^2 -k^2 -i0)^{-1} \big).
\end{align*}
The first part of the proof shows that 
\begin{align*}
\big \langle \mathcal{F}^{-1} \big((\xi^2 -k^2 -i0)^{-1} \big), \varphi \big  
\rangle  
&= \lim_{\epsilon \to 0+} \langle \mathcal{F}^{-1}\widehat{G}_{\epsilon}, 
\varphi \big
\rangle  \\
&= \lim_{\epsilon \to 0+} \int \frac{e^{i\sqrt{k^2+i\epsilon}|x|}}{4 \pi |x|}
\varphi(x) dx \\
&= \int \frac{e^{i k}|x|}{4 \pi |x|} \varphi(x) dx,
\end{align*}
so that $G_0(x) = e^{ik|x|}/(4\pi |x|)$.

To check condition \eqref{eq:GreensF}, we first note that $G_\epsilon$ is a 
Green
function, and hence 
\begin{align}
\lim_{\epsilon \to 0+}\big  \langle (-\Delta - k^2 -i \epsilon) G_\epsilon,
\varphi \big \rangle = \big \langle \delta, \varphi \big 
\rangle.\label{eq:eGFlim}
\end{align}
On the other hand we have
\begin{align*}
\lim_{\epsilon \to 0+}
\big \langle (-\Delta - k^2 -i \epsilon) G_\epsilon, \varphi \big \rangle 
& = \lim_{\epsilon \to 0+} \bigg( \big \langle (-\Delta - k^2 ) G_\epsilon, 
\varphi \big \rangle
+ \big \langle -i \epsilon G_\epsilon, \varphi \big \rangle \bigg) \\
& =  \big \langle G_0, (-\Delta - k^2 ) \varphi \big \rangle
+ \lim_{\epsilon \to 0+} \big \langle -i \epsilon G_\epsilon, \varphi \big 
\rangle.
\end{align*}
For the last  term we  have  $-i\epsilon \langle G_\epsilon , \varphi
\rangle \to  0$, as $\epsilon \to 0$. Combining the above with
\eqref{eq:eGFlim} gives
\begin{align*}
\big \langle (-\Delta - k^2 ) G_0, \varphi \big \rangle  = \big \langle
\delta, \varphi \big \rangle,
\end{align*}
so that \eqref{eq:GreensF} holds.

A direct calculation shows that $G_0$ satisfies the
Sommerfeld radiation condition \eqref{eq:SRC}.
Finally let us set $G_0(x,y) = G_0(x-y)$.
\end{proof}

\begin{rem} \label{remGF}
Notice that $G_0 \in \Ll^2(\R^3)$.  Since $G_0$ satisfies \eqref{eq:GreensF} in
the sense of distributions,  we have for $\varphi \in \Lc^2(\R^3)$
\begin{equation*}
(-\Delta -k^2)  (G_0 * \varphi) = \delta * \varphi = \varphi.
\end{equation*}
\end{rem}
Remark \ref{remGF} motivates the notation,
\begin{eqnarray*}
(-\Delta-k^2-i0)^{-1} \varphi(x) := G_0 * \varphi,
\end{eqnarray*}
where $-i0$ marks the fact that we took the positive sided limit in
\eqref{eq:regdef}. This is important since the positive limit guarantees that
$G_0$ and $G_0*\varphi$ are outgoing, i.e. satisfy the Sommerfeld condition. We 
would
have ended up with an \textit{incoming solution}, had we taken
the negative sided limit.

The above operator is in fact a limiting value of the  resolvent of the Laplace 
operator.
This operator is not continuous on $L^2(\R^3)$, when $\Im k =0$. 
It is however continuous between certain weighted spaces (see e.g.
\cite{Agmon1}).

The next lemma gives a more modest continuity result, which will be
used later.

\begin{lem} \label{cont1}
The operator 
\[
(-\Delta-k^2-i0)^{-1}:  \Lc^2(\R^3)\to \Ll^2(\R^3)
\]
is continuous. 
\end{lem}

\begin{proof}
Let us write  $T:=(-\Delta-k^2-i0)^{-1}$.  
Let  $B_r$ be a ball centered at the origin of  radius $r>0$ and $\chi_{B_r}$ 
stand
for the  characteristic function of the ball $B_r$. It suffices to show that 
the operator $\chi_{B_r}T\chi_{B_r}$ is continuous on $L^2(\R^3)$.  
We have, for $x\in B_r$, 
\begin{eqnarray*}
T\chi_{B_r} \varphi(x) = \int_{B_r} G_0(x-y) \varphi(y) dy, \quad \varphi\in L^2(\R^3).
\end{eqnarray*}
Since $G_0 \in L^2(B_r \times B_r)$ we get by applying the Cauchy--Schwarz
inequality, that
\begin{align*}
\int_{B_r}|T\chi_{B_r} \varphi(x)|^2 dx 
&\le  \int_{B_r} \bigg(\int_{B_r}G_0(x-y)\varphi(y)dy\bigg)^2dx \\
&\le  \int_{B_r} \int_{B_r} |G_0(x-y)|^2dy \;\|\varphi\|^2_2 dx \\
&\le  C \|\varphi\|^2_2.
\end{align*}
Hence $\| \chi_{B_r}T\chi_{B_r}\varphi\|_2 \leq C \| \varphi\|_2$.

\end{proof}

Next we use the closed graph theorem and elliptic regularity to extend the above
result to case where the range is the Sobolev space $H^2_{loc}(\R^3)$.

\begin{cor} \label{cont3}
The operator $(-\Delta-k^2-i0)^{-1} :L^2_{comp}(\R^3) \nuoli 
H^2_{loc}(\R^3)$ is continuous.
\end{cor}
\begin{proof} Let us write $T:=(-\Delta-k^2-i0)^{-1}$. By elliptic 
regularity, we have
\[
TL^2_{comp}(\R^3)\subset H^2_{loc}(\R^3). 
\]
The claim follows from the closed graph theorem, once we show that
the operator $T:L^2(B)\to H^2(B)$ is closed. Here $B$ is an open 
ball in $\R^3$. 
To that end, let $f_n\in L^2(B)$ be such that $f_n\to f$ in $L^2(B)$ and 
$Tf_n\to g$ in $H^2(B)$.  By Lemma \ref{cont1}, we know that $Tf_n\to Tf$ in 
$L^2(B)$. Hence, $g=Tf$. This completes the proof. 
\end{proof} 

We now begin investigating the asymptotics of radiating solutions to
the Helmholtz equation. First we look at the asymptotics of $G_0$.

\begin{lem} \label{asymFG}
The outgoing Green function has the following asymptotics,
\begin{equation}
G_0(x,y) = \frac{e^{ik(|x|-(x\cdot y)/|x|)}}{4 \pi |x|} + O\bigg( 
\frac{1}{|x|^2}\bigg),
     \label{eq:asymFG}
\end{equation}
as $|x| \nuoli \infty$, uniform for $y$ in a bounded set. It can be 
differentiated with respect to $y$. 
\end{lem}
\begin{proof} We will derive the asymptotic expression estimating the
nominator and the reciprocal of the denominator of $G_0$ separately.

We start by writing
\begin{eqnarray*}
|x-y| = |x|\sqrt{1 -\frac{2x \cdot y}{|x|^2} + \frac{|y|^2}{|x|^2}}. 
\end{eqnarray*}
By looking at the Taylor series we see that $\sqrt{1+x} = 1 + x/2 + O(x^2)$, as
$x \to 0$. Applying this to the above gives
\[
|x-y| = |x| -\frac{x \cdot y}{|x|} + O\bigg(  \frac{1}{|x|} \bigg), 
\]
for $|x| \nuoli \infty$.  It follows that 
\begin{equation}
\frac{1}{|x-y|} = \frac{1}{|x|} + O\bigg(  \frac{1}{|x|^2} \bigg),
\label{eq:as1}
\end{equation}
as $|x| \nuoli \infty$.

For the nominator of $G_0$, let $f(x)=O(1/|x|)$ be the function, for which 
$ik|x-y| = ik|x| -ik{x \cdot y}/|x| + f(x)$. Then
\begin{align*}
\exp(ik|x-y|) 
& = 
\exp \bigg( ik|x| -ik\frac{x \cdot y}{|x|} \bigg) \exp ( f(x)) \\
& = 
\exp \bigg( ik|x| -ik\frac{x \cdot y}{|x|} \bigg) \bigg( 1 + f(x) +
\frac{f(x)^2}{2!} + \dots \bigg),
\end{align*}
which gives the expression
\begin{align*}
\exp(ik|x-y|) & = 
\exp \bigg( ik|x| -ik\frac{x \cdot y}{|x|} \bigg) + O \bigg( \frac{1}{|x|} 
\bigg), 
\end{align*}
as $|x| \nuoli \infty$.  Multiplying this with the asymptotic expression 
\eqref{eq:as1} gives the asymptotic expression for $G_0$.
\end{proof}

The next small lemma shows that the $L^2$ norm over a sphere of an outgoing 
solution to the 
Helmholtz equation stays bounded as the radius of the sphere grows. Note that 
this applies also to
incoming solutions. This will be used later.
The lemma is mainly needed, since we use a weak form of the Sommerfeld
radiation condition \eqref{eq:SRC}. 

\begin{lem} \label{SRChlp}
Let $u \in \Hl^2(\R^3)$ be an outgoing or incoming solution to the Helmholtz 
equation $(-\Delta-k^2)u=0$ in $\R^3\setminus\ov{B_r}$, where  
$B_r=\{x\,\big|\,|x|<r\}$. Then
\begin{align}
\|u\|_{L^2(\partial B_s)}=O(1), \quad \|\partial_n u\|_{L^2(\partial B_s)} = 
O(1), \quad{s \nuoli \infty}. \label{eq:ab1}
\end{align}
\end{lem}
\begin{proof} 
Pick a ball $B_s \supset \ov{B_r}$. 
Assume that $u$ is outgoing, i.e.  it satisfies \eqref{eq:SRC}.
Multiplying this condition with its 
complex conjugate and integrating over $B_s$ gives
\begin{align}
&\int_{|x|=s}  \big(k^2|u |^2 + |\partial_n u |^2 + 2k \Im(u \partial_n \ov{u} 
) \big)dS=\int_{|x|=s} |\p_n u-iku|^2dS \nuoli 0,
\label{eq:lim1}
\end{align}
as $s \nuoli \infty$.  Here $n$ is the unit outer normal to $\p B_s$.  Using 
Green's formulas we have that
\begin{align*}
\int_{\partial B_s \cup \partial B_r} (u \partial_n \ov{u}  -
\ov{u}  \partial_n u )dS
= \int_{B_s \setminus B_r} (u \Delta \ov{u}  - \ov{u} 
\Delta u)dx  =0,
\end{align*}
and therefore, 
\begin{align*}
\int_{\partial B_r}  (u \partial_n \overline{u}  - \overline{u} \partial_n u )dS
&=
-\int_{\partial B_s} ( u \partial_n \overline{u}  - \overline{u}  \partial_n u 
)dS \\
&= -2i\int_{\partial B_s} \Im(u \partial_n \overline{u} )dS.
\end{align*}
It follows that the last expression does not depend on $s$. 
This shows that the two first terms in \eqref{eq:lim1}, which are both
positive,  are bounded,  and particularly we see that \eqref{eq:ab1} holds.

The incoming case can be reduced to the outgoing case as follows.
Firstly if 
$u$ is an incoming solution of the Helmholtz equation, then $\ov{u}$ is an outgoing 
solution. Moreover $\|u\|_{L^2}=\|\ov{u}\|_{L^2}$. 
The incoming case follows therefore from the outgoing case.
\end{proof} 

The following lemma gives a boundary integral 
representation for radiating solutions of the Helmholtz equation in 
exterior regions. 
\begin{lem} \label{repr1}
Let $u \in \Hl^2(\R^3)$ be an outgoing solution to the Helmholtz equation 
$(-\Delta-k^2)u=0$ in $\R^3\setminus\ov{B_r}$, where $B_r=\{x\in\R^3\,\big|\,|x|<r\}$. Then 
\begin{eqnarray*}
u(x) = \int_{\partial B_{r_1}} 
(\partial_{n_y} G_0(x,y) u(y) - G_0(x,y) 
\partial_{n_y} u(y) )dS(y),
\end{eqnarray*}
for $x \in \R^3 \setminus \overline{B}_{r_1}$ and $r_1>r$.
\end{lem}
\begin{proof} 
Let $r<r_1<r_2$ and $x_0\in \R^3$ be an arbitrary point in 
$\Omega := B_{r_2}\setminus\ov{B_{r_1}}$. 
Applying the Green formula to 
$\Omega_\epsilon:= B_{r_2}\setminus(\ov{B_{r_1}}\cup\ov{B_\epsilon(x_0)})$, $\epsilon>0$
sufficiently small, we get
\begin{equation}
\label{eq_green_rep_1}
\begin{aligned}
0=&\int_{\Omega_\epsilon} (-\Delta-k^2)G_0(x,x_0)u(x)
- G_0(x,x_0)(-\Delta-k^2)u(x)dx\\
=& \int_{\p \Omega_\epsilon} (-\p_nG_0(x,x_0)u(x)+G_0(x,x_0)\p_n u(x))dS(x). 
\end{aligned}
\end{equation}
We have 
\begin{align*}
\bigg| \int_{\p B_\epsilon(x_0)}G_0(x,x_0)\p_n u(x)dS(x) \bigg|\le \int_{\p B_\epsilon(x_0)}\frac{|\p_n u(x)|}{4\pi\epsilon}dS(x)\le O(\epsilon).
\end{align*}
Consider next
\begin{align*}
\int_{\p B_\epsilon(x_0)}\p_nG_0(x,x_0)u(x) dS(x)=&-\int_{\p B_\epsilon(x_0)} \frac{ike^{ik\epsilon}}{4\pi\epsilon} u(x)dS(x)\\
&+ \int_{\p B_\epsilon(x_0)} \frac{e^{ik\epsilon}}{4\pi\epsilon^2}u(x)dS(x). 
\end{align*}
It follows that 
\[
\bigg|\int_{\p B_\epsilon(x_0)} \frac{ike^{ik\epsilon}}{4\pi\epsilon} u(x)dS(x)\bigg|=O(\epsilon),
\]
as $\epsilon \to 0$.
Using the fact that $u(x)=u(x_0)+ O(\epsilon)$, we get
\[
\int_{\p B_\epsilon(x_0)} \frac{e^{ik\epsilon}}{4\pi\epsilon^2}u(x)dS(x)=e^{ik\epsilon}u(x_0)+O(\epsilon). 
\]
Letting $\epsilon\to 0$ in \eqref{eq_green_rep_1}, we obtain that 
\begin{align*}
u(x_0)=&\int_{\p \Omega} (-\p_nG_0(x,x_0)u(x)+G_0(x,x_0)\p_n u(x))dS(x).
\end{align*}
The next step is to show that 
\[
I:=\int_{\p B_{r_2}} (-\p_nG_0(x,x_0)u(x)+G_0(x,x_0)\p_n u(x))dS(x)\to 0,
\]
as $r_2\to\infty$. 
By adding and subtracting  $ikG_0u$, $I$ can be written as 
\begin{align*}
 I=&-\int_{\p B_{r_2}} (\p_nG_0(x,x_0)-ikG_0(x,x_0))u(x)dS(x)\\
 &+\int_{\p B_{r_2}}  G_0(x,x_0)(\p_n u(x) -iku(x)) dS(x).
  \end{align*}
We show that the first term goes to zero as $r_2 \nuoli \infty$. The second term
can be estimated in the same way. Because of \eqref{eq:ab1} we can use the
Cauchy-Schwarz inequality  and write
\begin{align*}
  \bigg|\int_{\partial B_{r_2}} u (\partial_n G_0 - ikG_0) dS\bigg|^2
  \leq \int_{\partial B_{r_2}} |u|^2 dS 
  \int_{\partial B_{r_2}} |\partial_n G_0 - ikG_0|^2 dS 
  \nuoli 0,
\end{align*}
as $r_2 \nuoli \infty$. Here we have used that $|\partial_n G_0 -
ikG_0|^2 = o(r_2^{-2})$, valid because of \eqref{eq:SRC}.

\end{proof}

The preceding lemmas allow us to prove the main result on the asymptotics of 
outgoing
solutions to the Helmholtz equation.

\begin{prop} \label{asymHE}
Let $u \in \Hl^2(\R^3)$ be an outgoing solution to the Helmholtz equation  
$(-\Delta-k^2)u=0$ in $\R^3\setminus\ov{B_r}$, where $B_r=\{x\in\R^3\,\big|\,|x|<r\}$.
Then there exists $a \in L^2(S^2)$ such that 
\begin{eqnarray}
u(x) = \frac{e^{ik|x|}}{|x|} a(\hat{x}) + O\bigg( \frac{1}{|x|^2}\bigg), \quad 
\hat{x} := x/|x|  \in S^2, \label{eq:asymHE}
\end{eqnarray}
uniformly in all directions as $x \nuoli \infty$.
\end{prop}
\begin{proof}
An application of  Lemma \ref{asymFG} allows us to write 
\[
G_0(x,y) = \frac{e^{ik|x|}}{|x|} a_1(\hat{x},y) + f_1(x,y), 
\]
where $f_1= O(1/|x|^2)$ uniformly in all directions as $|x| \nuoli \infty$.
Next, a straightforward computation using Lemma \ref{asymFG} shows that 
\[
\partial_{n_y} G_0(x,y) = \frac{e^{ik|x|}}{|x|} a_2(\hat{x},y) + f_2(x,y),
\]
where $f_2 = O(1/|x|^2)$ uniformly in all directions as $|x| \nuoli \infty$.
By the representation of Lemma \ref{repr1} we have
\begin{align*}
u(x) = \int_{\partial B_{r_1}} 
& \bigg( \frac{e^{ik|x|}}{|x|} a_2(\hat{x},y) + f_2(x,y) \bigg)u(y)dS(y) \\
-
& \int_{\partial B_{r_1}} \bigg(
\frac{e^{ik|x|}}{|x|} a_1(\hat{x},y) +  f_1(x,y) \bigg) \partial_n u(y) dS(y),
\end{align*}
with $r_1>r$. 
Let us split this into four separate integrations  corresponding to the 
individual terms. The terms
involving $a_1$ and $a_2$ are clearly of the form of the first term on the 
right 
side of the equation \eqref{eq:asymHE}. The two remaining terms give the 
contribution,
\begin{align*}
\int_{\partial B_{r_1}}  (f_2(x,y)u(y) - f_1(x,y)\partial_n u(y)) dS(y).
\end{align*}
Since $f_1,f_2 = O(1/|x|^2)$ uniformly in all directions as $|x| \nuoli 
\infty$,  we conclude that  the integral above is  $O(1/|x|^2)$, which proves 
the claim.
\end{proof}

The function $a$ in \eqref{eq:asymHE} is called \textit{the far field
pattern} or \textit{scattering amplitude} and is of central interest in
scattering theory.

\subsection{The magnetic Schr\"odinger equation in $\R^3$}
\label{sec:dirWS}

In this subsection we begin proving existence and uniqueness for 
problem
\eqref{eq:BVP}, \eqref{eq:SRC} by first considering the partial differential 
equation
in the whole of $\R^3$. The main tool
will be the Lax-Phillips method from scattering theory, see \cite{LaxPhillips1} 
and
\cite{Isakov1}.
We will assume that the potentials and sources are compactly supported in 
$\R^3$. More
precisely,  we assume that 
\begin{align}
A\in\Wc^{1,\infty}(\R^3,\R^3),\quad 
q\in\Lc^{\infty}(\R^3,\C),\quad \Im q\le 0,
\text{ and }  f \in \Lc^2(\R^3). \label{eq:assumR1}
\end{align}
Our aim is to find an outgoing solution $u$ to 
\begin{eqnarray}
    (\Ld - k^2 ) u & = f, \quad \text{in $\R^3$} \label{eq:probR}
\end{eqnarray}
and to show that it is unique.

We begin by recalling what is meant by a weak solution.
We call $u \in H^1_{loc}(\R^3)$ a \textit{weak solution} to \eqref{eq:probR}
if for every $v \in C^\infty_0(\R^3)$
\begin{eqnarray}
\int_{\R^3} (\nabla u \cdot \nabla \overline{v}  - 2iA\cdot (\nabla u) 
\overline{v} 
+ P_{k,q,A} u \overline{v} )dx
= \int_{\R^3} f\overline{v}dx, \label{eq:ws}
\end{eqnarray}
where $P_{k,q,A}  := -i\nabla \cdot A + A^2 + q - k^2$.
We state explicitly the following regularity result for weak solutions.
\begin{lem} \label{reg1}
Let $u$ be a weak solution to \eqref{eq:probR}. Then $u \in H^2_{loc}(\R^3)$.  
\end{lem}
\begin{proof} Since $u \in \Hl^1(\R^3)$ and $f \in L^2(\R^3)$, we have 
\begin{eqnarray*}
 -\Delta u = 2i A \cdot \nabla u + P_{k,q,A} u + f \in L_{loc}^2(\R^3).
\end{eqnarray*}
By elliptic regularity, see \cite{Grubb1}, we conclude that $u \in 
H^2_{loc}(\R^3)$. 
\end{proof}

The asymptotics given by Proposition \ref{asymHE}, combined with
Rellich's lemma and the unique continuation principle, see the appendix,  give 
the uniqueness 
for outgoing solutions to \eqref{eq:probR}.  This is the content of the 
following
theorem.

\begin{thm} \label{uniq1}
Assume that $A$ and $q$ satisfy \eqref{eq:assumR1}.
If $u \in
\Hl^{1}(\R^3)$ satisfies the Sommerfeld condition \eqref{eq:SRC} and solves
\begin{equation*}
(\Ld - k^2) u = 0 \quad \text{in}\quad \R^3,
\end{equation*}
where $k > 0$, then $u \equiv 0$.
\end{thm}
\begin{proof} 
Let $B_r :=B(0,r)$.
Denote the $L^2(B_r)$ inner product by $(\cdot,\cdot)$. It follows from the
Green's formula of Lemma
\ref{MagGFI}, in the appendix that
\begin{align*}
\quad &( (L_{A,q} -k^2)u, u) - (u,(L_{A,0}-k^2)u) -(u,\ov{q} u ) \\
&= (u , (\p_n + iA\cdot n) u)_{L^2(\p B_r)}-  ((\p_n + iA\cdot n) u, u)_{L^2(\p
B_r)}.
\end{align*}
The first term on the left side vanishes and the second term is $(u,qu)$.
Notice also that the vector field $A$ vanishes along $\p B_r$. 
The above equation reduces thus to 
\begin{eqnarray*}
\Im \int_{\partial B_r}  \overline{u} \partial_n u dS= \Im \int_{B_r} q|u|^2dx\le 0.
\end{eqnarray*}
Using the asymptotic expansions in Proposition \ref{asymHE}, we get
\begin{eqnarray*}
& &\Im \int_{\partial B_r} 
\bigg( \overline{a}\frac{e^{-ik|x|}}{|x|} + O\bigg(\frac{1}{|x|^2}\bigg) \bigg)
\bigg( a\frac{ike^{ik|x|}}{|x|} + O\bigg(\frac{1}{|x|^2}\bigg)
\bigg)dS \\
&=& \Im \int_{\partial B_r} \bigg(|a|^2\frac{ik}{|x|^2} + O\bigg(\frac{1}{|x|^3}\bigg)\bigg)dS \\
&=& \Im \int_{|x|=1}\bigg(|a(\theta,\varphi)|^2\frac{ik}{r^2} + O\bigg(\frac{1}{r^3}\bigg)\bigg)
r^2sin\theta d\theta d\varphi.
\end{eqnarray*}
By taking the limit as $r \nuoli \infty$, we obtain that 
\begin{eqnarray*}
\int_{|x|=1} k |a(\theta,\varphi)|^2 sin\theta d\theta d\varphi \leq 0,
\end{eqnarray*}
and hence the far-field pattern $a$ vanishes identically. 

An application of Rellich's lemma, see Proposition \ref{RL} in  the appendix,
allows us to conclude that
$u \equiv 0$ outside $B_r$. The unique continuation principle, see Theorem
\ref{UCP} in the appendix,  implies that $u \equiv 0$ in $\R^3$.
\end{proof}

We now proceed to proving the existence of outgoing solutions to 
\eqref{eq:probR}.
In doing so, we shall first establish that the Dirichlet realization of $\Ld$ 
on a ball has a discrete spectrum.
Showing this is complicated by the presence of the imaginary part 
of $q$, which  makes $\Ld$ non-self-adjoint. We will use the notion of
relative compactness for operators to resolve this issue.

\begin{defn}
Let  $\mathcal{B}$ be a Banach space and let $T$ be a closed densely defined 
linear
operator on $\mathcal{B}$ such that $\spec(T) \neq \C$.  Assume that $A$ is
a linear operator on $\mathcal{B}$ such that  $D(T) \subset D(A)$. We say that 
$A$ is \textit{relatively
compact}
with respect to $T$ if for any sequence $\{u_n\} \subset D(T)$, such that both
$\{u_n\}$ and $\{Tu_n\}$ are bounded, $\{Au_n\}$ has a convergent subsequence.
\end{defn}

\begin{lem} \label{rcmpct} Let $\mathcal{B}$ be a Banach space, let $T$ be a 
closed densely defined linear
operator on $\mathcal{B}$, and assume that  $\lambda \notin \spec(T)$. Then 
$A$ is relatively compact with respect to $T$ if and only if 
$A(\lambda-T)^{-1}$ is
compact. 
\end{lem}
\begin{proof}  Suppose that  we have a sequence $\{u_n\} \subset D(T)$ such 
that  
$\{u_n\}$ and $\{Tu_n\}$ are bounded. Then there is a constant $M$ such that
\begin{eqnarray*}
\| (\lambda - T) u_n\| \leq |\lambda|\|u_n\| + \|Tu_n\| < M < \infty,
\end{eqnarray*}
for all $n$, so that $\{ (\lambda - T) u_n\}$ is bounded.
Since  $A(\lambda-T)^{-1}$ is compact, the sequence 
\begin{eqnarray*}
\{ A(\lambda - T)^{-1} (\lambda - T) u_n \} = \{ Au_n \},
\end{eqnarray*}
has a convergent subsequence.

A similar deduction shows the opposite direction of the lemma.
\end{proof} 

We will now show that the Dirichlet realization of $\Ld$ on a ball  has a 
discrete
spectrum.

\begin{lem} \label{spec1}
Assume that $A$ and $q$ satisfy \eqref{eq:assumR1} and that
\[
\supp(q),\, \supp(A) \subset  B_r=\{x\in \R^3 \,\big|\, |x|<r\}.
\] 
The operator 
$\Ld:L^2(B_r)\nuoli L^2(B_r)$, equipped with the domain $H^2(B_r) \cap 
H^1_0(B_r)$,
is closed and has  discrete spectrum.
\end{lem}
\begin{proof} 
We let $L_0$ be the  operator $-\Delta$ on $L^2(B_r)$, equipped with the 
domain 
$H^2(B_r) \cap H^1_0(B_r)=:D(L_0)$.
We know that $L_0$ is self-adjoint with discrete spectrum, see \cite{Grubb1}.

We will show that $\Ld$ is a relatively compact perturbation of
$L_0$.  According to Theorem 11.2.6 in \cite{Davies1}, the essential spectrum 
is preserved under relatively compact
perturbations,  as is closedness.  Thus $\Ld$ will have no essential spectrum, 
since  $L_0$ has none.   By writing 
\[
L_{A,q}=-\Delta-2iA\cdot \nabla + p,\quad p=-i\nabla\cdot A+A^2+q\in 
L^\infty(B_r),
\]
we see that our task is  reduced to showing that $-2iA\cdot\nabla+p$ is
relatively compact with respect to  $L_0$.

Assume that $\lambda \notin \spec(L_0)$. By the criterion 
of Lemma \ref{rcmpct}, the operator $-2iA\cdot\nabla+p$ is relatively
compact with respect to $L_0$ if and only if  
\begin{equation*}
(-2iA\cdot\nabla+p)(\lambda I -L_0)^{-1}
\end{equation*}
is compact on $L^2(B_r)$.  We split this as
\begin{equation*}
-2iA\cdot\nabla(\lambda I -L_0)^{-1} + p(\lambda I -L_0)^{-1},
\end{equation*}
and show that both of the resulting operators are compact on 
$L^2(B_r)$.  
The resolvent operator $(\lambda I -L_0)^{-1} $ is continuous:  $L^2(B_r) 
\nuoli H^2(B_r) \cap
H^1_0(B_r)$. The latter space is however compactly imbedded into $L^2(B_r)$, by
the Rellich-Kondrachov theorem.
If we view $p$ as a multiplication operator, then it is 
continuous: $L^2(B_r) \nuoli L^2(B_r)$. This shows that the
second operator is compact.

The operator $2iA \cdot \nabla$ is continuous from $H^2(B_r) \nuoli
H^1(B_r)$, since $A$ is Lipschitz. The latter space is however compactly 
imbedded into $L^2(B_r)$.
The first operator is therefore also compact.

\end{proof}

We are now ready to prove the existence of   outgoing solutions to 
\eqref{eq:probR}, 
using the Lax-Phillips method.

\begin{thm} \label{LPmethod}
Assume that $A$ and $q$ satisfy \eqref{eq:assumR1}. Let $k>0$. 
Then  for any  $f \in \Lc^2(\R^3)$, there exists 
$u \in H^2_{loc}(\R^3)$, satisfying the Sommerfeld radiation condition,  
that solves
\begin{equation*}
(\Ld - k^2) u = f\quad \textrm{in}\quad \R^3. 
\end{equation*}
\end{thm}
\begin{proof} 
Let $B_r=\{x\in\R^3\,\big |\,|x|<r\}$ be such that 
$\supp(q), \supp(A) \subset  B_r$. Let $s>r$ be such that $\supp(f)\subset 
B_s$.  
We pick a function $\varphi \in C_0^{\infty}(\R^3,[0,1])$ such that  $\varphi = 
1$ on $B_r$ and
 $\supp(\varphi) \subset B_s$.

Let $\lambda \in \C$, $\Im \lambda \neq 0$, be such that  $\lambda$ avoids the 
spectrum of the Dirichlet realization of $\Ld$ on the ball $B_s$. 
The existence of such $\lambda$ is guaranteed  by Lemma \ref{spec1}.

We begin by looking for a solution $u$ of the form
\begin{equation}
u = \varphi w + (1-\varphi) v \label{eq:sform}.
\end{equation}
Here $w \in H^2(B_s) \cap H^1_0(B_s)$ is the 
unique solution to the Dirichlet problem
\begin{eqnarray*}
 (\Ld-\lambda)w &=& g \hspace{5mm} \text{in $B_s$} ,\\
 w|_{\partial B_s} &=& 0,
\end{eqnarray*}
where $g \in L^2(\R^3)$, with $\supp(g) \subset B_s$.
And $v$ is the unique outgoing solution of the equation,
\begin{equation}
 (-\Delta - k^2)v = g \quad \text{ in $\R^3$}. \label{eq:LPmethod1}
\end{equation}
Remark \ref{remGF} gives an explicit formula for $v$, and according to 
Corollary \ref{cont1}, we know that $v \in \Hl^2(\R^3)$.  

Inserting $u$ into the original equation will result in an operator equation
for the unknown function $g$. Abbreviating $L:=(\Ld - k^2)$, we 
have\footnote{Here
we use the following bracket notation for the
commutator operator $[A,B] := AB - BA$.}
\begin{eqnarray*}
L u &=& L(\varphi w) + L((1-\varphi)v) \\
&=& [L,\varphi]w + \varphi L w  + [L,(1-\varphi)]v + (1-\varphi) L v \\
&=& [L,\varphi]w + \varphi (\Ld -\lambda)w + \varphi(\lambda-k^2)w + 
[L,(1-\varphi)]v + (1-\varphi)g \\
&=& [L,\varphi]w + \varphi g + \varphi(\lambda-k^2)w + [L,(1-\varphi)]v + 
(1-\varphi)g
\end{eqnarray*}
Noting that the commutator of $k^2$ and $\varphi$ is zero, we get from the
above,
\begin{eqnarray*}
L u &=& [\Ld,\varphi]w + \varphi(\lambda-k^2)w + [\Ld,(1-\varphi)]v + g \\
    &=& [\Ld,\varphi]w + \varphi(\lambda-k^2)w - [\Ld,\varphi]v + g.
\end{eqnarray*}
By setting $Tg := \varphi(\lambda-k^2)w+[\Ld, \varphi](w-v)$, we see that $g$ 
is to
satisfy the operator equation
\begin{eqnarray}
(I+T)g=f. \label{eq:fred1}
\end{eqnarray}
Our problem of finding a solution of the special form \eqref{eq:sform} is thus 
reduced 
to showing that the equation \eqref{eq:fred1} has a solution.

Our aim is to use the Fredholm theory. In order to do this we need to show that 
$T:L^2(B_s) \nuoli L^2(B_s)$ is compact. 
Notice that the first term of $T$ is 
$\varphi(\lambda-k^2)(\Ld-\lambda)^{-1}g$.
But since the resolvent $(\Ld-\lambda)^{-1}:L^2(B_s) \nuoli H^2(B_s)$ is bounded and
since the inclusion map: $H^2(B_s) \to L^2(B_s)$ is compact, we see that the 
first term of $T$ is compact on $L^2(B_s)$. 

For the second term of $T$, we first note that
$(-\Delta -k^2-i0)^{-1}:L^2(B_s) \nuoli H^2(B_s)$ is
bounded by Corollary \ref{cont3}. Next note that 
$\supp([\Ld,\varphi]v) \subset B_s$, since $\supp(\varphi) \subset B_s$. The 
first order 
operator
$[\Ld,\varphi]$ is explicitly  given by 
\begin{eqnarray*}
[\Ld,\varphi]  = -\Delta \varphi - 2 \nabla \varphi \cdot \nabla - 2iA\cdot
\nabla \varphi.
\end{eqnarray*}
It also has $W^{1,\infty}$ coefficients, so that it maps $H^2(B_s)$ to
$H^1(B_s)$ continuously. Now $H^1(B_s)$ is compactly embedded
in $L^2(B_s)$ and hence we see that $T$ is compact on $L^2(B_s)$.

According to the Fredholm theory,  we need only to show that $I+T$ is 
injective, to
have the surjectivity and thus a solution to \eqref{eq:fred1}, for a given $f$. 

Assume that $(I+T)g=0$. Theorem \ref{uniq1} gives that 
\begin{eqnarray}
u = \varphi w + (1- \varphi)v = 0 \quad \text{in $\R^3$}. \label{eq:vweq}
\end{eqnarray}
We want to show that $g \equiv 0$. This follows if $w \equiv 0$, which we will
prove next.
First note that \eqref{eq:vweq} gives $\varphi w = (\varphi - 1)v$, 
so that 
\begin{equation}
\label{eq_LP1}
w=0\quad\textrm{when}\quad \varphi=1.
\end{equation}
 In particular, $w=0$ on $B_r$. This gives that
\[
g=(\Ld - \lambda) w = (-\Delta -\lambda)w 
 \]
in $B_s$, because $\supp(A), \supp(q)  \subset B_r$. We also have that
\begin{eqnarray*}
 (-\Delta -k^2)v &=& g,
\end{eqnarray*}
in $B_s$. Subtracting the last two equations and using that $v=\varphi(v-w)$, 
we get 
\begin{eqnarray*}
  (-\Delta -k^2)v + (-\Delta -\lambda)(-w)
 &=&  (-\Delta -\lambda)(v-w) - k^2v + \lambda v \\
 &=&  (-\Delta -\lambda)(v-w) - (k^2 - \lambda)\varphi (v-w) \\
 &=& 0.
\end{eqnarray*}
Set $r := v-w$. Now $(-\Delta -\lambda)r - (k^2 - \lambda)\varphi r = 0$
in $B_s$. Multiplying this by $\overline{r}$, using integration by parts and
noting that $r|_{\partial B_s} = -w |_{\partial B_s} = 0$, because of
\eqref{eq:vweq} we get
\begin{eqnarray*}
\int_{B_s} |\nabla r|^2 - \lambda |r|^2dx =
\int_{B_s} (k^2- \lambda) \varphi|r|^2dx.
\end{eqnarray*}
Taking the imaginary part  gives
\begin{eqnarray*}
  \Im \lambda \int_{B_s} (1-\varphi)|r|^2dx= 0.
\end{eqnarray*}
From this we see that
$r=0$ in the region where $\varphi\ne 1$. It follows that $v=\varphi r=0$ when 
$\varphi\ne 1$. 
Hence, 
\begin{equation}
\label{eq_LP2}
w=v-r=0\quad\textrm{when}\quad \varphi\ne 1. 
\end{equation}
Combining \eqref{eq_LP1} and \eqref{eq_LP2}, we see that $w=0$ in $B_s$.  

\end{proof}

In summary, we see that Theorem \ref{uniq1} and Theorem \ref{LPmethod} show the 
 existence and 
uniqueness of outgoing solutions for the problem \eqref{eq:probR}.

\subsection{The magnetic Schr\"odinger equation in a half space}
\label{sec:dirHS}

The main objective of this subsection is to extend the existence and uniqueness
results of the previous subsection to the half space case.  

Let $k>0$ and let 
\begin{equation}
\label{eq_sec2_0}
A\in \Wc^{1,\infty}(\CHS,\R^3), \quad q\in \Lc^\infty(\CHS,\C),\quad \Im 
q\le 0.
\end{equation}
Given $f\in\Lc^2(\CHS,\C) :=\{f\in L^2(\CHS,\C)\;\big|\,\supp(f)\subset \CHS \text{ compact}\}$.
we first prove the existence and uniqueness 
of an outgoing solution to the following problem,
\begin{equation}
\label{eq_sec2_1}
\begin{aligned}
&(L_{A,q}-k^2)u=f\quad\textrm{in}\quad\HS,\\
&u|_{\p \HS}=0.
\end{aligned}
\end{equation}
We will reduce this problem to the case of $\R^3$, by using an extension argument
and then use Theorems \ref{uniq1} and \ref{LPmethod}.
The extension argument relies on the following Lemma, which will also be of
importance later (see also Theorem 1.3.3 in \cite{Horm_book_1}).

\newtheorem{BorelL}[thm]{Lemma}
\begin{BorelL} \label{BorelL}
Let $v \in C^{\,0,1}(\R^2)$, with compact support.
Then there is a $\psi \in C^{1,1}(\R^3)$, with compact support for which 
\begin{equation} \label{eq_Fcrit}
\psi(x,0) = 0 \quad \text{ and }  \quad\partial_3 \psi(x,0) = v(x),
\end{equation}
for $x \in \R^2$.
\end{BorelL}
\begin{proof} Let $\varphi$ be the usual mollifier function in $\R^2$, 
with 
$\varphi \in C^\infty_0(\R^2)$, $\varphi \geq 0$ and $\int \varphi =1$.
Set
$\varphi_t(x) = 1/t^2\varphi(x/t)$. We define $u(x,t):= (v * \varphi_t)(x)$, for
$t\neq0$ and $u(x,0):= v(x) = (v*\delta)(x)$. More explicitly  
\begin{align} \label{eq_udef}
u(x,t) = \frac{1}{t^2} \int_{\R^2} v(y) \varphi\Big( \frac{x-y}{t} \Big)dy
= \int_{\R^2} v(x-ty) \varphi(y) dy,
\end{align}
for $t\neq0$. From the right hand side we see that $u$ is Lipschitz in $(x,t)$,
because $v$ is Lipschitz.
We define $\Psi$ as
\begin{align} \label{eq_Fdef} 
\Psi(x,t) := tu(x,t).
\end{align}
Notice that $\Psi$ satisfies the first condition in \eqref{eq_Fcrit}.

Next we show that the partial derivatives of $\Psi$ are Lipschitz.
When $t\neq0$
we have, using \eqref{eq_udef} that
\begin{align*}
\p_{x_i}(tu) &= t \int_{\R^2} v(y) (\p_i \varphi)\Big( \frac{x-y}{t}
\Big)\frac{1}{t^3} dy \\
&=   \int_{\R^2} v(y) (\p_i \varphi)\Big( \frac{x-y}{t}
\Big)\frac{1}{t^2} dy \\
&= \int_{\R^2} v(x-ty) \p_i \varphi(y) dy.
\end{align*}
It follows that this identity also holds when $t=0$.
The right hand side of this identity is easily seen to 
be Lipschitz in both $x$ and $t$, since $v$ is
Lipschitz. It follows that $\p_{x_i}(tu)$ is Lipschitz in $\R^3$.

The next step is to show that $\p_t \Psi$ is Lipschitz. 
To see that $\p_t\Psi=\p_t (tu)$ is continuous, we compute the derivative at $t=0$
\begin{align*}
\p_t (tu)|_{t=0} =\lim_{h\to0} \frac{hu(x,h)-0\,u(x,0)}{h} = u(x,0) = v(x),
\end{align*}
Notice that this shows that  $\Psi$ also satisfies the second 
condition in \eqref{eq_Fcrit}. And further that $\p_t (tu)$ is Lipschitz in $x$
when $t=0$.

For $t\neq0$,  observe firstly that
$\p_t(tu) = u + t\p_t u$. Thus we need only to check that the later term is
Lipschitz. We write using \eqref{eq_udef}
\begin{align} \label{eq_ptu} 
\p_t u 
= -\int_{\R^2} v(x-ty) \Big( \nabla \varphi(y) \cdot \frac{y}{t} + \varphi(y)
\frac{2}{t} \Big) dy.
\end{align} 
This gives that
\begin{align*} 
t\p_t u 
= -\int_{\R^2} v(x-ty) \big(\nabla \varphi(y) \cdot y + 2\varphi(y) \big) dy,
\end{align*} 
which is easily seen to be Lipschitz in both $x$ and $t$, because $v$ is
Lipschitz.

To obtain $\psi$ we pick a $\chi \in \C^\infty_0(\R^3)$, s.t. $\chi|\supp(v)=1$.
Then $\psi:=\chi\Psi \in C^{1,1}(\R^3)$, $\chi\Psi|_{t=0}=0$ and
$\p_t(\chi\Psi)|_{t=0}= (\chi \p_t \Psi)|_{t=0} = v$.

\end{proof}

We use the above Lemma to show that, it is sufficient to consider existence and
uniqueness in problem
\eqref{eq_sec2_1}, for potentials $A$ for which $\p_3 A|_{ x_3 = 0 }=0$. To
see this notice that Lemma \ref{BorelL} guarantees the existence of a 
$\psi\in C^{1,1}(\R^3,\R)$ with compact support, for which
$\psi|_{x_3=0}=0$ and $\nabla \psi|_{x_3=0} = (0,0,-A_3)|_{x_3=0}$. 
A straight forward computation shows that 
\[
e^{-i\psi}L_{A,q} e^{i\psi}=L_{A+\nabla\psi,q},
\]
Using this
we see that $u$ is an outgoing solution to the problem \eqref{eq_sec2_1} if and 
only if $\tilde u=e^{-i\psi}u$ is an outgoing solution to the problem,
\begin{align*}
&(L_{A+\nabla\psi,q}-k^2)\tilde 
u=e^{-i\psi}f\quad\textrm{in}\quad\HS,\\
&\tilde u|_{\p \HS}=0.
\end{align*}
We can thus, without loss of generality, assume that $A_3=0$ 
along $\p\HS$, when showing that the solution of problem of  \eqref{eq_sec2_1} 
exists and is unique. 

We have the following result. 
\begin{thm} \label{HS1}
Let $A$ and $q$ satisfy \eqref{eq_sec2_0} 
Then  for any  $f \in \Lc^2(\CHS)$, there exists a unique outgoing solution
$u \in H^2_{loc}(\ov{\HS})$
to the problem \eqref{eq_sec2_1}. 
\end{thm}

\begin{proof}

We shall first prove the existence. 
In doing so, we shall reduce the problem \eqref{eq_sec2_1} to all of $\R^3$ by 
making use of appropriate even and odd extensions of the coefficients of the 
operator $L_{A,q}$. 
The discussion preceding the Lemma \ref{BorelL} shows that we may without loss of generality,
assume that $A_3=0$  along $\p\HS$. 

We extend the potentials $A=(A_1,A_2,A_3)$ and $q$,  and the source term 
$f$ to the whole of $\R^3$. Let $\tilde x := (x_1,x_2,-x_3)$.
For $A_1$, $A_2$, and $q$, we do even extensions in $x_3$, i.e.,
\begin{align*}
\tilde A_j(x)&=\begin{cases}  A_j(x),& x_3<0,\\
 A_j(\tilde x),& x_3>0,
\end{cases}\quad j=1,2,\\
\tilde q(x)&=\begin{cases}  q(x),& x_3<0,\\
 q(\tilde x),& x_3>0.
\end{cases}
\end{align*}
For $A_3$ and $f$, we do odd extensions in $x_3$, 
\begin{align*}
\tilde A_3(x)&=\begin{cases}  A_3(x),& x_3<0,\\
- A_3(\tilde x),& x_3>0,
\end{cases}\\
\tilde f(x)&=\begin{cases}  f(x),& x_3<0,\\
 -f(\tilde x),& x_3>0.
\end{cases}
\end{align*}
Since $A_3=0$ when $x_3=0$, we see that $\tilde A\in \Wc^{1,\infty}(\R^3,\R)$. 
Furthermore, $\tilde q\in \Lc^\infty(\R^3)$ and $\tilde f\in\Lc^2(\R^3)$.

By Theorem \ref{uniq1} and Theorem \ref{LPmethod}, the problem 
\[
(L_{\tilde A,\tilde q}-k^2)\tilde u=\tilde f\quad\textrm{in}\quad \R^3
\]
has a unique outgoing solution $\tilde u\in H^2_{loc}(\R^3)$. 

Next we want to show that $\tilde u$ is odd in $x_3$. To that end it is 
convenient to write, 
\begin{equation}
\label{eq_LAQ_1}
L_{\tilde A,\tilde q}=-\Delta -2i\tilde A\cdot\nabla + \tilde p, \quad \tilde 
p=-i\nabla\cdot \tilde A +\tilde A^2+\tilde q. 
\end{equation}
Here one sees easily that the operators $\Delta$, $\tilde A_3\partial_3$ and
$\tilde p$ all preserve the parity in $x_3$ of a function that they operate on. 
Hence, the operator 
$L_{\tilde A,\tilde q} -k^2$
preserves the parity in $x_3$.

Decompose $\tilde u$ into an even and odd
part with respect to $x_3$, i.e. 
\[
\tilde u=\tilde u_e+\tilde u_o,
\]
where
\begin{align*}
\tilde u_e(x)=\frac{1}{2}  \big(\tilde u(x)+\tilde u(\tilde x) \big), \quad
\tilde u_o(x)=\frac{1}{2}  \big(\tilde u(x)-\tilde u(\tilde x) \big).
\end{align*}
Then 
\[
\tilde f=(L_{\tilde A,\tilde q} -k^2)\tilde u_e+(L_{\tilde A,\tilde q} 
-k^2)\tilde u_0,
\]
and using that $\tilde f$ is odd with respect to $x_3$, we conclude that 
\[
(L_{\tilde A,\tilde q} -k^2)\tilde u_e=0\quad \textrm{in}\quad \R^3. 
\]
Now a direct computation shows that the function $x\mapsto \tilde 
u(\tilde x)$ is outgoing on $\R^3$, since $\tilde u$ has this property. 
Thus, $\tilde u_e$ is outgoing, and by Theorem \ref{uniq1}, $\tilde u_e=0$. 
Hence, $\tilde u$ is odd in $x_3$. 

The Sobolev embedding theorem shows that $\tilde u$ is continuous in
$\R^3$, since $\tilde u \in \Hl^2(\R^3)$. Hence,  $\tilde u|_{\partial \HS}=0$, 
so that $\tilde u|_{\HS}$ is a solution to the half space Dirichlet problem 
\eqref{eq_sec2_1}.

In order to prove uniqueness, we assume that $u\in H^2_{loc}(\CHS)$ 
is an outgoing solution to  the problem  \eqref{eq_sec2_1} with $f=0$. 
We need to show that $u \equiv 0$.
To that end, let us consider the odd extension of $u$ with respect to $x_3$, 
i.e.
\begin{equation}
\label{eq_tilde-u}
  \tilde u(x) = \begin{cases}
    u(x), & \quad x_3 < 0, \\
    -u(\tilde x), & \quad x_3 > 0.\\
  \end{cases}
\end{equation}
Notice that since  $u=0$ along $x_3=0$, the function $\tilde u$ is continuous 
across $x_3=0$.    

Let us now show that $\tilde u$ satisfies the equation,
\begin{equation}
\label{eq_LAQ_5}
(L_{\tilde A,\tilde q}-k^2)\tilde u=0\quad\textrm{in}\quad \R^3,
\end{equation}
with $\tilde A$ and $\tilde q$ as in the first part of the proof.  Indeed, 
computing the first order partial derivatives of $\tilde u$, given by 
\eqref{eq_tilde-u}, in the sense of distributions on  $\R^3$,  we obtain that   
 
\begin{equation}
\label{eq_LAQ_2}
\begin{aligned}
\p_j \tilde u(x)&=\begin{cases}
    (\p_j u)(x), & \quad x_3 < 0, \\
    -(\p_ju)(\tilde x), & \quad x_3 > 0,\\
  \end{cases}\quad j=1,2, \\
  \p_3\tilde u(x)&=\begin{cases}
    (\p_3u)(x), & \quad x_3 < 0, \\
    (\p_3u)(\tilde x), & \quad x_3 > 0.\\
  \end{cases}
\end{aligned}
\end{equation}
Hence, we see that $\tilde u\in H^1_{\textrm{loc}}(\R^3)$.  

One has to be more careful when computing the second order partial derivatives 
of  $\tilde u$.  For this reason, we shall give the details of the computation 
below. Let $\varphi\in C^\infty_0(\R^3)$. Then denoting by 
$\langle\cdot,\cdot\rangle$ the duality between distributions and test 
functions,  $x'=(x_1,x_2)$, and $\p_{x'}^2=\p_1^2+\p_2^2$,  we have  
\begin{align*}
\langle-\Delta \tilde u,\varphi\rangle=-\int_{\R^3}\tilde u\Delta \varphi 
dx=-\int_{-\infty}^0\int_{\R^2} u(x',x_3)(\p_{x'}^2+\p_{3}^2)\varphi dx'dx_3\\
+\int_{0}^{+\infty}\int_{\R^2} u(x',-x_3)(\p_{x'}^2+\p_{3}^2)\varphi dx'dx_3\\
=-\int_{-\infty}^0\int_{\R^2} (\p_{x'}^2u)(x',x_3)\varphi dx'dx_3+ 
\int_{0}^{+\infty}\int_{\R^2} (\p_{x'}^2u)(x',-x_3)\varphi dx'dx_3\\
+ \int_{\R^2} I(x')dx',
\end{align*}
where
\begin{align*}
I(x')=&-\int_{-\infty}^0u(x',x_3)\p_3^2\varphi dx_3+\int_0^{+\infty} 
u(x',-x_3)\p_3^2\varphi dx_3\\
=&
\int_{-\infty}^0 \p_{3} u(x',x_3)\p_{3}\varphi dx_3- 
u(x',x_3)\p_{3}\varphi|_{-\infty}^0\\
&+ \int_{0}^{+\infty} (\p_{3}u)(x',-x_3)\p_{3}\varphi dx_3+ 
u(x',-x_3)\p_{3}\varphi|_{0}^{+\infty}\\
=&-\int_{-\infty}^0 \p_{3}^2 u(x',x_3)\varphi dx_3+ \p_{3} u(x',x_3)\varphi 
|_{-\infty}^0\\
+ &\int_{0}^{+\infty} (\p_{3}^2u)(x',-x_3)\varphi dx_3+ 
(\p_{3}u)(x',-x_3)\varphi |_{0}^{+\infty}\\
=&-\int_{-\infty}^0 \p_{3}^2 u(x',x_3)\varphi dx_3+ \int_{0}^{+\infty} 
(\p_{3}^2u)(x',-x_3)\varphi dx_3.
\end{align*}
Hence, we have 
\begin{equation}
\label{eq_LAQ_3}
-(\Delta \tilde u)(x) = \begin{cases}
    -(\Delta u)(x), & \quad x_3 < 0, \\
    (\Delta u)(\tilde x), & \quad x_3 > 0,\\
  \end{cases}
\end{equation}
in the sense of distributions. 
Using \eqref{eq_LAQ_1}, \eqref{eq_LAQ_2} and \eqref{eq_LAQ_3}, we obtain that 
\begin{align*}
(L_{\tilde A,\tilde q}-k^2)\tilde u= \begin{cases}
    -\Delta u(x) -2iA(x)\cdot\nabla u(x)+(\tilde p(x)-k^2)u(x), & \quad x_3 < 
0, \\
    (\Delta u)(\tilde x) + 2iA(\tilde x)\cdot\nabla u(\tilde x)+ (-\tilde 
p(\tilde x)+k^2) u(\tilde x), & \quad x_3 > 0.
  \end{cases}
\end{align*}
We have therefore verified that $\tilde u$ solves \eqref{eq_LAQ_5}. 

Taking into account the fact that $\tilde u$ is outgoing and applying Theorem 
\ref{uniq1}, we finally get $u\equiv 0$ in $\HS$. 

\end{proof}

It is now straightforward to establish the main result of this section, Theorem 
\ref{thm_1_direct}. 

\textbf{Proof of Theorem \ref{thm_1_direct}.} 
Uniqueness follows from Theorem \ref{HS1}, since it implies that $v_1-v_2\equiv0$ for
any two solutions $v_1$ and $v_2$ of problem \eqref{eq:BVP} and \eqref{eq:SRC}.

It remains to check the existence of a solution. To that end, given $f\in 
H^{3/2}_{comp}(\p \HS)$, let $F\in H^2_{comp}(\R^3)$ be such that 
$F|_{\partial \HS} = f$. 
Theorem  \ref{HS1} gives the existence of an outgoing solution $v$ to
\begin{equation*}
 \left.\begin{aligned}
        (L_{A,q}-k^2 )v & = -(L_{A,q}-k^2 ) F,\\
        v |_{\partial \HS} &= 0.
       \end{aligned}
 \right.
\end{equation*}
Then $u=v+F$ solves the problem \eqref{eq:BVP} and satisfies \eqref{eq:SRC}.  
The proof of Theorem \ref{thm_1_direct} is complete. 
\begin{flushright}$\Box$
\end{flushright}

\newpage
\section{The inverse problem: Proof of Theorem \ref{thm_2_inverse}}
\label{sec:invprob}

The main task of this section is to prove Theorem \ref{thm_2_inverse}. 
It will be convenient to set
\[
\Bm :=\HS\cap B,\quad l:=\p \HS \cap B, 
\]
where $B\subset\R^3$ is the open ball of Theorem \ref{thm_2_inverse}
containing the supports of $A_j$ and $q_j$, $j=1,2$.
Recall that we assume that 
\[
(\p\HS  \setminus \ov{B}) \cap \Gamma_j \neq \emptyset, \quad j=1,2.
\]
We can thus choose $\tilde \Gamma_j$, such that 
\begin{align*}
\tilde \Gamma_j\subset\Gamma_j,\quad \tilde \Gamma_j \subset \subset 
\p\HS \setminus \ov{B} ,\quad j=1,2. 
\end{align*}
Then it follows from \eqref{eq_data_inv} that 
\begin{equation}
\label{eq_data_inv_1}
\Lambda_{A_1,q_1}(f)|_{\tilde \Gamma_1}=\Lambda_{A_2,q_2}(f)|_{\tilde \Gamma_1},
\end{equation}
for any $f\in H^{3/2}(\p \HS)$, $\supp(f)\subset \tilde \Gamma_2$.  In order 
to prove Theorem \ref{thm_2_inverse} we shall only  use the data 
\eqref{eq_data_inv_1}, which turns out to be enough to determine the magnetic 
field and the electric potential.

The gauge invariance of the DN-map plays an important role in 
the sequel. We state therefore the following results.

\begin{lem} \label{gauge_inv} 
Let $A \in W^{1,\infty}(\CHS,\R^3)$ and $q \in L^{\infty}(\CHS)$. Then
\begin{enumerate}[(i)]
\item For all $\psi \in C^{1,1}(\CHS,\R)$ we  have 
\[
e^{-i \psi}\Lambda_{A,q} e^{i\psi} =\Lambda_{A+\nabla\psi,q}.
\]
\item  There exists $\psi \in C^{1,1}(\CHS,\R)$ with $\psi|_{\{x_3=0\}} = 0$,
for which
\[
\Lambda_{A,q} = \Lambda_{A+\nabla\psi,q}
\]
and $(A+\nabla\psi)|_{\{x_3=0\}} = (A_1,A_2,0)$.
\end{enumerate}
\end{lem}
\begin{proof} \textit{(i).} A straight forward computation shows that  
\begin{align}
e^{-i \psi}L_{A,q} e^{i\psi} = L_{A+\nabla\psi,q}. \label{eq_Linv}
\end{align}
So that
a function $u$ solves $L_{A+\nabla\psi,q} u = 0$ if and only if $v=e^{i\psi}u$
solves $L_{A,q}v = 0$. Moreover we have
\begin{align*}
(e^{-i \psi}\Lambda_{A,q} e^{i\psi}) f &= e^{-i \psi} (\p_n + in\cdot A) 
(e^{i\psi}u) \\
&=   i\p_n \psi u +  \p_n u + in\cdot A u \\
&= \Lambda_{A+\nabla\psi,q}f.
\end{align*}
\textit{(ii).}  By Lemma \ref{BorelL} there exists a $\psi \in C^{1,1}(\CHS,\R)$ with
$\psi|_{\{x_3=0\}} = 0$ and $\nabla \psi|_{\{x_3=0\}} = (0,0,-A_3)$. By part
\textit{(i)}  we have
\begin{align*}
\Lambda_{A+\nabla\psi,q} f &=  e^{-i\psi} \Lambda_{A,q}  (e^{i\psi}) f \\
&=  \Lambda_{A,q} f.
\end{align*}
\end{proof}

\begin{rem} \label{rem_gauge}
Notice that part \textit{(ii)} of the Proposition says in other words
that  we can change a potential $A$ to $\tilde{A} = (A_1,A_2,A_3-\p_3 \psi) =
(A_1,A_2,0)$, while still retaining that $\Lambda_{A,q}=\Lambda_{\tilde{A},q}$. It
follows that we can always assume that $n\cdot A|_{\{ x_3=0\}} = 0$, where $n$
is unit normal to the plane $\{x_3=0\}$, without altering the DN-map.

Notice also that this is the reason why we can take the DN-map as having the
value $(\p_n+ i n\cdot A)u |_{\{x_3=0\}}$, instead of just $\p_n u|_{\{x_3=0\}}$,
even though we do not "know" $A$ on the boundary.
\end{rem}

\subsection{An integral identity}
\label{sec:IA}

One central step in the ideas that are used in proving uniqueness results for 
Calderon's problem, is to derive an integral equation that
expresses $L^2$
orthogonality between the product of two solutions $u_1$ and $u_2$,  and  the 
difference of two
potentials $q_1$ and $q_2$, see \cite{Uhl_review_1999}. One shows that
\begin{align*}
\int (q_1-q_2)u_1 u_2=0,
\end{align*}
provided that the DN-maps for $q_1$ and $q_2$ are equal.

A similar thing will be done in this subsection, for the magnetic case. The
integral equation, is however more involved in the case of a magnetic potential 
and cannot by
itself be interpreted as an orthogonality relation. 
We will be considering the integral equation in conjunction with solutions that
depend on a small positive parameter $h$. In the later sections we will see 
that 
in the limit $h \nuoli 0$, we obtain a criterion for the curl being zero.

We now begin deriving the integral identity. We 
assume that $A_j,q_j$ and $\Gamma_j$ 
are as in Theorem \ref{thm_2_inverse} and that 
\begin{align*}
\Lambda_{A_1,q_1}(f)|_{\Gamma_1}=\Lambda_{A_2,q_2}(f)|_{\Gamma_1},
\end{align*}
for any $f\in H^{3/2}_{\emph{\textrm{comp}}}(\p \HS)$, $\supp(f)\subset 
\Gamma_2$, 
so that \eqref{eq_data_inv_1} also applies.

Let $u_1\in 
H^2_{\textrm{loc}}(\overline{\HS})$ be the radiating solution to 
\begin{align*}
(L_{A_1,q_1} -k^2) u_1 &= 0, \text{ in $\HS$},  \\
u_1|_{\partial \HS} &= f, 
\end{align*}
with $f \in H^{3/2}(\p \HS)$, $\supp(f)\subset\tilde \Gamma_2$. Let $v\in 
H^2_{\textrm{loc}}(\overline{\HS})$ be the radiating solution to 
\begin{align*}
(L_{A_2,q_2} -k^2) v &= 0, \text{ in $\HS$}, \\
v|_{\partial \HS} &= f.
\end{align*}
Define $w:=v-u_1$. Then
\begin{align}
(L_{A_2,q_2} -k^2) w
&= 2i(A_2-A_1)\dotc \nabla u_1 + i \nabla \dotc(A_2-A_1)u_1 \nonumber \\
& \quad +(A_1^2-A_2^2)u_1 + (q_1-q_2)u_1. \label{eq:aoI}
\end{align}
It follows from \eqref{eq_data_inv_1}that
\begin{align}
\label{eq_sec3_1}
(\partial_n + i A_1\cdot n) u_1|_{\tilde \Gamma_1} = (\partial_n + i A_2\cdot 
n)v|_{\tilde \Gamma_1}. 
\end{align}

By Remark \ref{rem_gauge} we may assume that 
$A_1\cdot n=A_2\cdot n=0$ on $\p\HS$, so that $\p_n w=0$ on 
$\tilde \Gamma_1$.  
We also conclude from \eqref{eq:aoI} that $w$ satisfies the equation
\[
(-\Delta-k^2)w=0\quad \textrm{in}\quad \HS\setminus\ov{\Bm}. 
\]
As $w|_{\tilde \Gamma_1}=\p_n w|_{\tilde \Gamma_1}=0$, by unique continuation, we get that 
$w=0$ in $\HS \setminus \ov{\Bm}$.
See Theorem \ref{UCP} and 
Corollary \ref{UCP_boundary} in the appendix.
Since $w\in H^2_{\textrm{loc}}(\overline{\HS})$, we have 
\[
w=\p_n w=0\quad\textrm{on} \quad \p \Bm\cap \HS.
\]
Let $u_2\in H^2( \Bm)$ be a solution to $(L_{A_2,\overline{q_2}}-k^2)u_2 = 0$ 
in $ \Bm$.
Then by Green's formula, we get 
\begin{align*}
((L_{A_2,q_2}-k^2) w, u_2)_{L^2( \Bm)} 
&= (w , (L_{A_2,\overline{q_2}}-k^2) u_2)_{L^2( \Bm)} \\
&\quad - ((\partial_n + iA_2\cdot n) w, u_2)_{L^2(\partial  \Bm)} \\
&\quad + (w , (\partial_n + iA_2\cdot n) u_2)_{L^2(\partial  \Bm)} \\
&= -(\partial_n w, u_2)_{L^2(l)}.
\end{align*}
Assuming that 
\[
u_2=0\quad \textrm{on}\quad l,
\]
we conclude that 
\[
((L_{A_2,q_2}-k^2) w, u_2)_{L^2( \Bm)}=0. 
\]
Using equation \eqref{eq:aoI} we may write this as follows,  
\begin{align*}
\int_{\Bm} (2i(A_2-A_1)\dotc (\nabla u_1)\ov{u_2} + i \nabla \dotc(A_2-A_1)u_1 
\ov{u_2})\,dx\\
+ \int_{\Bm}(A_1^2-A_2^2 + q_1-q_2)u_1\ov{u_2}\,dx= 0.
\end{align*}
Using the fact that $(A_2-A_1)\cdot n=0$ on $\p  \Bm$ and  an integration by 
parts, we get 
\begin{align*}
i \int_{\Bm} \nabla \dotc(A_2-A_1)u_1 \ov{u_2}\,dx
= -i \int_{\Bm} (A_2-A_1) \dotc (\nabla u_1 \overline{u_2} 
+ u_1 \nabla \overline{u_2})dx.
\end{align*}
Thus, we obtain that  
\begin{equation}
\label{eq_sec3_3}
\begin{aligned}
\int_{ \Bm}& i(A_2-A_1) \dotc (\nabla u_1 \ov{u_2} - u_1 \nabla 
\ov{u_2})\,dx\\ 
 &+
\int_{ \Bm} (A_1^2-A_2^2 + q_1-q_2)u_1\ov{u_2}\,dx = 0, 
\end{aligned}
\end{equation}
where $u_1\in W_1(\HS)$ and $u_2\in W_2^*( \Bm)$. Here
\begin{align*}
W_1(\HS):=\{u\in H^2_{\textrm{loc}}(\overline{\HS})\;\big|\;  
(L_{A_1,q_1}-k^2)u=0\textrm{ in }\HS, \\
\supp(u_1|_{\p\HS})\subset\tilde \Gamma_2, u\textrm{ radiating}\},
\end{align*}
and
\begin{align*}
W_2^*( \Bm):=\{u\in H^2( \Bm)\;\big|\; (L_{A_2,\ov{q_2}}-k^2)u=0\textrm{ in } 
 \Bm, u|_{l}=0\}.
\end{align*}

We shall next extend the integral identity \eqref{eq_sec3_3} to a richer class 
of solutions to the magnetic Schr\"odinger operators. To that end, let us 
introduce the following space of solutions,
\[
W_1( \Bm):=\{u\in H^2( \Bm)\;\big|\; (L_{A_1,q_1}-k^2)u=0\textrm{ in }  \Bm, 
u|_{l}=0\}.
\] 
The following Runge type approximation result is similar to those found in 
\cite{I1}, \cite{LU1} and \cite{KLU1}.

\begin{lem} \label{runge}

The space $V_1:=W_1(\HS)|_{ \Bm}$ is dense in $W_1( \Bm)$ in the 
$L^2( \Bm)$--topology.

\end{lem}

\begin{proof}

Suppose that $V_1$ is not dense in $W_1( \Bm)$.
First notice that $\text{span}(V_1)=V_1$ 
so that $\ov{V_1}$ is a linear subspace of $L^2( \Bm)$. Since
$V_1$ is not dense in $W_1( \Bm)$, we have a vector  $u_0 \in W_1( \Bm)$ 
such that $u_0 \notin \ov{V_1}$.
We can decompose  $u_0$ as $u_0=a+b$, where $a \in
\ov{V_1}$, $b \in \ov{V_1}^\bot$ and $b \neq 0$.
Let $T$ be the linear functional on $L^2( \Bm)$, defined by 
$T(x):= \text{proj}_{\ov{V_1}^\bot}(x)/ \|b\|_{L^2}$, 
where $\text{proj}_{\ov{V_1}^\bot}$ is the orthogonal
projection to $\ov{V_1}^\bot$. Now clearly
$\|T(u_0)\|_{L^2} = 1$ and $T|_{V_1} = 0$. 

By the Riesz representation theorem, there is  $g_T \in L^2( \Bm)$ that
corresponds to $T$. Extend $g_T$ by zero to the complement of $ \Bm$ 
in $\HS$. 
Let $U\in H^2_{\textrm{loc}}(\overline{\HS})$ be the incoming solution to 
\begin{align*}
(L_{A_1,\ov{q_1}} -k^2) U &= g_T \quad \text{in}\quad \HS,\\
U|_{\partial \HS} &= 0. 
\end{align*}
The existence of such a solution follows from Theorem \ref{thm_1_direct}, since 
we can find a $\tilde U$ outgoing with  $(L_{-A_1,q_1}-k^2)\tilde U=\ov{g_T}$ 
and $\tilde U|_{\partial \HS} = 0$.  It then suffices to take 
$U=\overline{\tilde U}$.

Now let $u \in W_1(\HS)$. Then because $T|_{V_1}=0$ and $\supp(g_T) \subset 
 \Bm$, we get 
by the Green's formula of Lemma \ref{MagGFII} that
\begin{align*}
0=(u,g_T)_{L^2(\HS)}
&= (u,(L_{A_1,\ov{q_1}}-k^2) U)_{L^2(\HS)} \\
&= ((L_{A_1,q_1}-k^2) u,U)_{L^2(\HS)} \\
&\quad - (u , (\partial_n +iA_1\cdot n) U)_{L^2(\partial \HS)} \\
&\quad + ( (\partial_n + iA_1\cdot n) u, U)_{L^2(\partial \HS)} \\
&= -( u,\partial_n U)_{L^2(\tilde \Gamma_2)}.
\end{align*}
Since the boundary condition $u|_{\tilde \Gamma_2}$ can be chosen  arbitrarily 
from
$C^\infty_0(\tilde \Gamma_2)$,
we get  that $\partial_n U|_{\tilde \Gamma_2} = 0$. Since $ U|_{\tilde 
\Gamma_2} = 0$, we 
apply the unique continuation
principle 
to conclude that  $U|_{\HS \setminus \ov{ \Bm}} = 0$. As $U \in 
H^2_{\textrm{loc}}(\overline{\HS})$, we have 
\[
U|_{\p \Bm \cap \HS}= \p_n U|_{\p \Bm\cap \HS}= 0.
\]

Now applying Green's formula and 
doing the same calculation as above for $u_0$ and $ \Bm$ instead of $u$ yields
\begin{align*}
(u_0,g_T)_{L^2( \Bm)}
&= (u_0,(L_{A_1,\ov{q_1}}-k^2) U)_{L^2( \Bm)} \\
&= ((L_{A_1,q_1}-k^2) u_0,U)_{L^2( \Bm)} \\
&\quad - (u_0 , (\partial_n +iA_1\cdot n) U)_{L^2(\partial  \Bm)} \\
&\quad + ( (\partial_n + iA_1\cdot n) u_0, U)_{L^2(\partial  \Bm)} \\
&=-(u_0 , \partial_n U)_{L^2( l)}=0.
\end{align*}
Here we have used that $u_0|_{l}=0$. 
It follows that $T(u_0) = 0$. This contradiction completes the proof.
\end{proof}

Since $(A_2-A_1)\cdot n=0$ on $\p \Bm$,  we can rewrite \eqref{eq_sec3_3} in 
the following form,
\begin{align*}
-\int_{ \Bm}& u_1 i \nabla \cdot((A_2-A_1) \ov{u_2})\,dx  -\int_{ \Bm} i   
(A_2-A_1)\cdot (u_1 \nabla \ov{u_2})\,dx\\ 
 &+
\int_{ \Bm} (A_1^2-A_2^2 + q_1-q_2)u_1\ov{u_2}\,dx = 0.
\end{align*}
Hence, an application of Lemma \ref{runge} implies that the integral identity  
\eqref{eq_sec3_3} is valid for any $u_1\in W_1( \Bm)$ and any $u_2 \in 
W_2^*( \Bm)$.

We summarize the discussion in this subsection in the following result.

\begin{prop} \label{identity} 
Assume that $A_j,q_j$ and $\Gamma_j$, $j=1,2$ 
are as in Theorem \ref{thm_2_inverse} and that the DN-maps satisfy
\begin{align}
\label{eq_ieDN}
\Lambda_{A_1,q_1}(f)|_{\Gamma_1}=\Lambda_{A_2,q_2}(f)|_{\Gamma_1},
\end{align}
for any $f\in H^{3/2}_{\emph{\textrm{comp}}}(\p \HS)$, $\supp(f)\subset \Gamma_2$.
Then 
\begin{equation}
\begin{aligned}\label{eq_sec3_4} 
\int_{ \Bm}& i(A_2-A_1) \dotc (\nabla u_1 \ov{u_2} - u_1 \nabla 
\ov{u_2})\,dx\\ 
 &+
\int_{ \Bm} (A_1^2-A_2^2 + q_1-q_2)u_1\ov{u_2}\,dx = 0, 
\end{aligned}
\end{equation}
for  $u_1\in W_1( \Bm)$ and any $u_2 \in W_2^*( \Bm)$.
\end{prop}
\begin{flushright}$\Box$
\end{flushright}

\begin{rem} \label{rem_ie}
Notice that the proof of Proposition \ref{identity} only uses the assumption
\eqref{eq_data_inv_1}, which follows from \eqref{eq_ieDN}. 
Proposition \ref{identity} holds therefore also under the weaker assumption.
\end{rem} 

The next step in proving Theorem \ref{thm_2_inverse} is to use the integral 
identity \eqref{eq_sec3_4} with $u_j$, $j=1,2$, taken to be special solutions, 
which are called complex geometric optics solutions.

\subsection{Complex geometric optics solutions}
Let $\Omega\subset \R^3$  be a bounded domain with $C^\infty$-boundary, and let 
$A\in W^{1,\infty}(\Omega,\R^3)$,  $q\in L^\infty(\Omega,\C)$.
The task of this subsection is to review the construction of  complex geometric 
optics solutions for  the magnetic Schr\"odinger equation,
\begin{equation}
\label{eq_Sch_lip}
L_{A,q}u=0\quad \textrm{in}\quad \Omega. 
\end{equation}
A complex geometric optics solution to \eqref{eq_Sch_lip} is a solution of the form
\begin{equation}
\label{eq_cgo_lip}
u(x,\zeta;h)=e^{x\cdot\zeta/h}(a(x,\zeta;h)+r(x,\zeta;h)),
\end{equation}
where $\zeta\in \C^3$, $\zeta\cdot\zeta=0$, $a$ is a smooth 
amplitude,  $r$ is a remainder, and $h>0$ is a small parameter.

In the case when $A\in C^2(\overline{\Omega})$ and $q\in L^\infty(\Omega)$, 
such solutions were constructed in \cite{DosSantos1} using the method of 
Carleman estimates, and the construction was extended to the case of less 
regular potentials in \cite{KnuSalo}, see also \cite{KLU1}.

Let $\varphi(x) =
\alpha \cdot x$, $\alpha \in \R^3$, $|\alpha|=1$. The fundamental role in the 
construction  of complex geometric optics solutions is played by the following 
Carleman estimate, 
\begin{equation}
\label{eq_Carleman_est}
h\|u \|_{H^1_{\textrm{scl}}(\Omega)}\le C\|e^{\varphi/h}h^2L_{A,q} 
e^{-\varphi/h} u\|_{L^2(\Omega)},
\end{equation}
valid for all $u\in C^\infty_0(\Omega)$ and $0<h\le h_0$, which was proved in 
\cite{DosSantos1} and \cite{KnuSalo}.  Here 
$\|u\|_{H^1_{\textrm{scl}}(\Omega)}=\|u\|_{L^2(\Omega)}+\|h\nabla 
u\|_{L^2(\Omega)}$. 
 For the convenience of the reader, we shall present a derivation of 
\eqref{eq_Carleman_est} in the appendix. 

Based on the estimate \eqref{eq_Carleman_est}, the following solvability result 
was established in \cite[Proposition 4.3]{KnuSalo}. See also the discussion in 
the appendix. 
\begin{prop}
\label{prop_solvability}
Let $A \in W^{1,\infty}(\Omega,\R^3)$, $q \in
L^{\infty}(\Omega,\C)$, $\alpha \in \R^3$, $|\alpha|=1$ and $\varphi(x) =
\alpha \cdot x$. Then there is  $C > 0$ and  $h_0 > 0 $ such that for all $h\in 
(0,h_0]$, and any $f\in L^2(\Omega)$, the equation
\[
e^{\varphi/h}h^2L_{A,q} e^{-\varphi/h} u=f\quad\textrm{in}\quad \Omega,
\]
has a solution $u\in H^1(\Omega)$ with 
\begin{align*}
\|u \|_{H^1_{\emph{\textrm{scl}}}(\Omega)}\le \frac{C}{h}\|f\|_{L^2(\Omega)}. 
\end{align*}
\end{prop}
\begin{proof} See appendix.
\end{proof}

Our basic strategy in constructing solutions of the form \eqref{eq_cgo_lip} 
is to write \eqref{eq_Sch_lip}, as
\begin{align}
\label{eq_split}
L_{\zeta} r = -L_{\zeta} a,
\end{align}
where $L_{\zeta} :=  e^{-x\cdot \zeta/h} h^2L_{A,q} e^{x\cdot \zeta/h}$. 
Then we first  search for a suitable $a$, after which we will get $r$ by
Proposition \ref{prop_solvability}. We must however take some care in choosing
$a$ and the way it depends on $h$, since we need later that 
$\| r \|_{H^1_{\emph{\textrm{scl}}}(\Omega)} \to 0$, sufficiently fast as $h \to 0$.
We need $a$ also to be smooth enough. This will be handled as in 
\cite{KnuSalo}.

We extend $A\in W^{1,\infty}(\Omega, \R^3)$ to a Lipschitz vector field, 
compactly supported  in $\tilde \Omega$, where $\tilde \Omega\subset \R^3$ is 
an open bounded set such that $\Omega\subset\subset \tilde \Omega$. We consider 
the mollification $A^\sharp := A*\psi_\epsilon\in
C_0^\infty(\tilde \Omega,\R^3)$.
Here $\epsilon>0$ is small and 
$\psi_\epsilon(x)=\epsilon^{-3}\psi(x/\epsilon)$ is the usual mollifier with 
$\psi\in C^\infty_0(\R^3)$, $0\le \psi\le 1$, and 
$\int \psi dx=1$.  We write 
$A^\flat=A-A^\sharp$. Notice that we have the following estimates for
$A^\flat$,
\begin{equation}
\label{eq_flat_est}
\|A^\flat\|_{L^\infty(\Omega)}=\mathcal{O}(\epsilon),
\end{equation}
\[
 \|\p^\alpha A^\sharp\|_{L^\infty(\Omega)}=\mathcal{O}(\epsilon^{-|\alpha|}) \quad 
\textrm{for all}\quad \alpha,
\]
as $\epsilon\to 0$. 

We shall work with a complex  $\zeta = \zeta_0+\zeta_1$ depending slightly on $h$, for which
\begin{align} \label{eq_zassum}
&\zeta\cdot \zeta = 0,\;
\zeta_0:= \alpha + i \beta,\; \alpha,\beta \in S^2,\; \alpha \cdot \beta=0, \\
&\zeta_0 \, \text{ independent of $h$ and }
\;\zeta_1=\mathcal{O}(h),\; \text{as $h\to 0$}. \nonumber
\end{align}
By expanding the conjugated operator we write the right hand
side of \eqref{eq_split} as
\begin{align}
\label{eq_La}
L_\zeta a 
=&(-h^2\Delta 
-2i(-i\zeta_0+hA)\cdot h\nabla -2\zeta_1\cdot h\nabla+h^2A^2 \nonumber \\
&-2ih\zeta_0\cdot (A^\sharp+A^\flat)-2ih\zeta_1\cdot A 
-ih^2(\nabla\cdot A)+h^2q)a. 
\end{align}
Now we want $a$ to be such that this expression decays more rapidly than
$\mathcal{O}(h)$, as $h \to 0$.

Consider the operator in \eqref{eq_La}, ignoring for the time being $a$ and its
possible dependence on $h$. We would like to eliminate from this operator
the terms that are first order in $h$.
Notice first that $\zeta_1 \in\mathcal{O}(h)$ and that we can control 
$\|A^\flat\|_{L^\infty(\Omega)}$ with $h$,
if we choose $\epsilon$ to be dependent on $h$.
Then in an attempt to eliminate first order terms in $h$, it is
natural to search for an $a$ for which
\begin{align}
\label{eq_az0eq}
\zeta_0 \cdot \nabla  a =  -i \zeta_0 \cdot A^\sharp  a, \quad \text{ in
$\Omega$}.
\end{align}

We will look for a solution of the form $a=e^\Phi$. The above equation becomes then
\begin{align} \label{eq_trans1}
\zeta_0 \cdot \nabla \Phi =  
-i \zeta_0 \cdot A^\sharp  , \quad \text{ in $\Omega$}.
\end{align}
Pick a
$\gamma \in S^2$, such that $ \gamma \bot \{  \alpha, 
\beta\}$. 

Next we consider the above equation in coordinates $y$, associated with the
basis $\{\alpha,\beta,\gamma\}$. Let $T$ be the coordinate transform
$y = Tx:=( x \cdot \alpha, x \cdot \beta, x \cdot \gamma)$.
Using the chain rule and the fact that $T^{-1} = T^*$, 
one gets that\footnote{Here $T^*$ is the transpose of $T$.}
\begin{align*}
\nabla(\Phi \circ T^{-1})(Tx) 
= T [\nabla \Phi(x)]^*.
\end{align*}
We therefore have that
\begin{align*}
(1,i,0) \cdot \nabla(\Phi \circ T^{-1})(Tx) 
&= (1,i,0) \cdot T [\nabla \Phi(x)]^* \\
&=
(\alpha \cdot\nabla+ i\beta\cdot\nabla)\Phi(x) \\
&= \zeta_0 \cdot \nabla\Phi(x).
\end{align*}
Equation \eqref{eq_trans1} gives hence the $\bar\p$-equation 
\begin{align} \label{eq_dbar}
2\p_{\bar z} \cdot (\Phi \circ T^{-1}) (y)   
 = -i \zeta_0 \cdot (A^\sharp \circ T^{-1})(y),
\end{align}
where $\p_{\bar z} = (\p_{y_1}+i\p_{y_2})/2$.
We will solve this using the Cauchy operator
\begin{align*}
N^{-1}f (x) := \frac{1}{\pi} \int_{\R^2} 
\frac{1}{s_1 + is_2} f(x - (s_1,s_2,0)) ds_1 ds_2,
\end{align*}
which is an inverse for the $\bar \p$-operator,
$N:= (\partial_{y_1} + i\partial_{y_2})/2$
(see e.g. \cite{Horm_book_2} Theorem 1.2.2).
We will need the following straight forward continuity result for the Cauchy
operator.

\newtheorem{cauchyop}[thm]{Lemma}
\begin{cauchyop} \label{cauchyop} 
Let $r>0$ and $f \in W^{k,\infty}(\R^3)$, $k\geq0$ and assume that $\supp(f) \subset
B(0,r)$. Then 
\begin{align*}
\|N^{-1}f\|_{W^{k,\infty}(\R^3)} \leq C_k\|f\|_{W^{k,\infty}(\R^3)} 
\end{align*}
for some constant $C_k>0$.
And if $f \in C_0(\R^3)$, then $N^{-1}f \in C(\R^3)$.
\end{cauchyop}
\begin{proof} See e.g. \cite{MS3}.
\end{proof}

Returning to \eqref{eq_dbar} we get that
$\Phi = \frac{1}{2}N^{-1} (-i \zeta_0 \cdot (A^\sharp \circ T^{-1})) \circ T$. Or more explicitly  that
\begin{align} 
\label{eq_Phi}
\Phi(x,\zeta_0;\epsilon)
&=
\frac{1}{2\pi} \int_{\R^2} 
\frac{-i \zeta_0 \cdot A^\sharp(x-T^{-1}(s_1,s_2,0))}{s_1+is_2}ds_1ds_2,
\end{align}
here $T^{-1}(s_1,s_2,0)=s_1\alpha+s_2\beta$.
We have thus found a solution $a = e^\Phi$ to equation \eqref{eq_az0eq}.
We will choose $\epsilon$ so that it depends on $h$, which implies that $a$ will
depend on $h$.
In order to determine how the norm of $r$ will depend on $h$ and also for later
estimates, we will need to see how $\|\p^\alpha a\|_{L_\infty}$ depends on $h$.

\begin{lem} Equation \eqref{eq_az0eq} has a solution $a\in
C^\infty(\overline{\Omega})$ satisfying the estimates
\begin{align}
\label{eq_ampl_est}
\|\p^\alpha a\|_{L^\infty(\Omega)}\le C_\alpha \epsilon^{-|\alpha|}\quad 
\textrm{for all}\quad \alpha. 
\end{align}
\end{lem}
\begin{proof} Existence of a solution is a consequence of the considerations
above. Therefore we need only to prove the norm estimate.

For $\alpha=0$, 
Lemma \ref{cauchyop} gives that $\|\Phi\|_{L^\infty(\Omega)}\leq C$. From
this it follows that $\|e^\Phi\|_{L^\infty(\Omega)}\leq C'$.
For $|\alpha|>1$ argue similarly using the estimates \eqref{eq_flat_est}.

\end{proof}

We can now write the $L^\infty(\Omega)$ norm of \eqref{eq_La} as
\begin{align*}
\|L_\zeta a\|_{L^\infty(\Omega)}= 
\|-h^2 L_{A,q}a +2ih\zeta_0\cdot A^\flat a+2\zeta_1\cdot h\nabla 
a+2ih\zeta_1\cdot A a\|_{L^\infty(\Omega)}.
\end{align*}
Using \eqref{eq_flat_est}, \eqref{eq_ampl_est} and the fact that 
$\zeta_1=\mathcal{O}(h)$ we have that 
\begin{align*}
\|L_\zeta a\|_{L^\infty(\Omega)} = \mathcal{O}(h^2\epsilon^{-2}+h\epsilon).
\end{align*}
Choosing $\epsilon = h^{1/3}$, gives finally $\|L_\zeta a\|_{L^\infty(\Omega)} =
\mathcal{O}(h^{4/3})$, as $h \to 0$.

Finally to solve \eqref{eq_split} for $r$, we rewrite it as  
\begin{align} \label{eq_req}
e^{-x\cdot \Re \zeta/h} h^2L_{A,q} e^{x\cdot \Re \zeta/h}(e^{ix\cdot \Im \zeta/h}r)
= -e^{ix\cdot \Im \zeta/h}L_\zeta a.
\end{align}
If we replace $e^{ix\cdot \Im \zeta/h}r$ by $\tilde r$, then
the solvability result \ref{prop_solvability}, shows that we can find a solution $\tilde r$,
so that a solution $r$ to \eqref{eq_req} is given by $r = e^{-ix\cdot \Im \zeta/h} \tilde r$.

To get a norm estimate for $r$, notice that for the left hand side of \eqref{eq_req}
we have
\begin{align*}
\|e^{ix\cdot \Im \zeta/h}L_\zeta a\|_{L^\infty(\Omega)}  
= \mathcal{O}(h^{4/3}),
\end{align*}
as $h \to 0$. The solvability result \ref{prop_solvability} gives then that
\begin{align*}
\| \tilde r \|_{H^1_{\emph{\textrm{scl}}}(\Omega)}
= \mathcal{O}(h^{1/3}),
\end{align*}
as $h \to 0$, which implies that $\|r\|_{H^1_{\emph{\textrm{scl}}}(\Omega)}  = \mathcal{O}(h^{1/3})$,
as $h \to 0$.

Thus we have obtained the following existence result for complex geometric optics
solutions.

\begin{prop}
\label{prop_CGO_Lip}
Let $A\in W^{1,\infty}(\Omega,\R^3)$ and $q\in L^\infty(\Omega,\C)$. Then for 
$h>0$ small enough, there exist solutions $u\in H^1(\Omega)$,
of the equation 
\[
L_{A,q}u=0\quad \textrm{in}\quad \Omega, 
\]
that are of the form 
\[
u(x,\zeta;h)=e^{x\cdot\zeta/h}(a(x,\zeta;h)+r(x,\zeta;h)),
\]
where $\zeta\in \C^3$, is of the form given by
\eqref{eq_zassum}, $a\in C^\infty(\overline{\Omega})$ solves the
equation \eqref{eq_az0eq},  and where $a$ and $r$ satisfy the estimates
\[
\|\p^\alpha a\|_{L^\infty(\Omega)}\le C_\alpha h^{-|\alpha|/3} 
\quad \text{and} \quad
\|r\|_{H^1_{\emph{\textrm{scl}}}(\Omega)}=\mathcal{O}(h^{1/3}). 
\]
\end{prop}
\begin{flushright}$\Box$
\end{flushright}

\begin{rem} 
\label{rem_com_geom_1} In the sequel, we  need complex geometric optics 
solutions belonging to $H^{2}(\Omega)$. To obtain such solutions, let 
$\Omega'\supset\supset\Omega$ be a bounded domain with smooth boundary,  and 
let us extend $A\in W^{1,\infty}(\Omega,\R^3)$ and $q\in L^\infty(\Omega)$ to 
$W^{1,\infty}(\Omega',\R^3)$ and $L^\infty(\Omega')$-functions, respectively. 
By elliptic regularity, the complex geometric optics solutions, constructed on 
$\Omega'$, according to Proposition \emph{\ref{prop_CGO_Lip}},   belong to  
$H^{2}(\Omega)$.

\end{rem}

\begin{rem} \label{rem_CGO_Lip}
Recall that $\Phi = \frac{1}{2}N^{-1} (-i (\alpha+i\beta) \cdot (A^\sharp \circ T^{-1})) \circ T$. Lemma \ref{cauchyop} implies that
$N^{-1}:C_0(\Omega)\to C(\Omega)$ is continuous. The estimates \eqref{eq_flat_est}
show that $A^\sharp \to A$ uniformly on $\Omega$. It follows that there is an
$\Phi^0$, s.t.
\[
\|\Phi(x,\zeta_0;h^{1/3})- \Phi^0 \|_{L^\infty(\Omega)}\to 0, \quad h\to 0,
\]
where $\Phi^0=\frac{1}{2}N^{-1} (-i (\alpha+i\beta) \cdot (A \circ T^{-1})) \circ T$ solves the equation
\begin{align}
\label{eq_Phi0}
\zeta_0 \cdot \nabla \Phi^0 =-i \zeta_0 \cdot A \quad \textrm{in}\quad \Omega, 
\end{align}
as $h \to 0$.
\end{rem}

\begin{rem}  \label{rem_g}
We shall later use a slightly more general form for the amplitude $a$ 
in the CGO solutions. Namely we suppose that $a =
g e^{\Phi}$, where $g \in C^{\infty}(\ov{\Omega})$, with 
\begin{align}
\zeta_0 \cdot \nabla g =0. \label{eq:gcond}
\end{align}
This means that $g$ is holomorphic in a plane spanned by $\alpha$ and
$\beta$.
Notice also that by picking $a=ge^\Phi$, we get by \eqref{eq_az0eq} that
\begin{align*}
 \zeta_0 \cdot g\nabla \Phi = -i \zeta_0 \cdot gA^\sharp,
\end{align*}
in place of \eqref{eq_trans1}. But the $\Phi$ solving \eqref{eq_trans1}
also solves the above. Hence we can use the same argument to obtain the $\Phi$
for the above equation, as earlier.

We thus obtain CGO solutions of the form
\begin{align*}
u = e^{x\dotc \zeta/h}(ge^{\Phi} + r_g),
\end{align*}
where $\Phi$ is the same as when $a$ is of the earlier of form with no $g$.

Notice also that setting $a=ge^{\Phi}$ does not affect the norm estimates on $a$
in Proposition \ref{prop_CGO_Lip}, since $g$ does not depend on $h$.
\end{rem}

\subsection{Recovering the magnetic field}

The aim of this section is to prove the first part of Theorem \ref{thm_2_inverse}, by
showing that the curl of the magnetic potential is
determined by the DN-map.
We use again similar notations as in Subsection \ref{sec:IA}, i.e. 
\[
\Bm:=\HS\cap B,\quad B_+:=\R^3_+\cap B ,\quad l:=\p\HS\cap B,
\]
where $B$ is an open ball in $\R^3$, containing the supports of the potentials 
$A_j$ and $q_j$, $j=1,2$.
The first step in the argument will be to 
construct complex geometric optics solutions $u_1$ and $u_2$, belonging to the 
spaces $W_1(\Bm)$ and $W_2^*(\Bm)$ (defined in Section \ref{sec:IA}) and then to
examine the limit of \eqref{eq_sec3_4} as $h \to 0$.

For $u_1\in W_1(\Bm)$ and $u_2\in W_2^*(\Bm)$, we have that 
$u_j|_{l}=0$, $j=1,2$. To obtain solutions that satisfy this condition, we will
first choose solutions defined on the bigger set $B = B_+ \cup l \cup \Bm $.

The parameters $\zeta$ for the complex geometric optics
solutions will be picked as follows. We will assume that   
\begin{align}
\label{eq_gamma-assum}
\xi,\gamma_1,\gamma_2\in\R^3,\; |\gamma_1|=|\gamma_2|=1\; 
\text{ and that }\{\gamma_1,\gamma_2,\xi\} \; \text{is orthogonal}. 
\end{align}
Similarly to \cite{S1}, we set 
\begin{align}
\label{eq_zeta_1_2}
\zeta_1&=\frac{ih\xi}{2}+\gamma_1+ i\sqrt{1-h^2\frac{|\xi|^2}{4}}\gamma_2, \\ 
\zeta_2&=-\frac{ih\xi}{2}-\gamma_1+i\sqrt{1-h^2\frac{|\xi|^2}{4}}\gamma_2, \nonumber 
\end{align}
so that $\zeta_j\cdot\zeta_j=0$, $j=1,2$, and 
$(\zeta_1+\overline{\zeta_2})/h=i\xi$. Here $h>0$ is a small 
semiclassical parameter.  

We need to extend the potentials $A_j$ and $q_j$, $j=1,2$, to  $B_+$.  For
the component functions 
$A_{j,1}$, $A_{j,2}$, and $q_j$, we do an even extension, and  for $A_{j,3}$, we 
do an odd extension, i.e.,  for $j=1,2$ we set,
\begin{align*}
\tilde A_{j,k}(x)&=\begin{cases} A_{j,k}(x),& x_3<0,\\
 A_{j,k}(\tilde x),& x_3>0,
\end{cases},\quad k=1,2,\\
\tilde A_{j,3}(x)&=\begin{cases} A_{j,3}(x),& x_3<0,\\
- A_{j,3}(\tilde x),& x_3>0,
\end{cases}\\
\tilde q_j(x)&=\begin{cases} q_j(x),& x_3<0,\\
 q_j(\tilde x),& x_3>0,
\end{cases}
\end{align*}  
where $\tilde x := (x_1,x_2,-x_3)$.
By Remark \ref{rem_gauge} we can take $A_{j,3}|_{x_3=0}=0$, from which it follows  
that  $\tilde A_j\in W^{1,\infty}(B)$ and 
$\tilde q_j\in L^{\infty}(B)$, $j=1,2$.

We can now by Proposition \ref{prop_CGO_Lip} and Remark \ref{rem_com_geom_1}
pick complex geometric optics solutions  $\tilde u_1$ in $H^2(B)$,
\[
 \tilde u_1(x,\zeta_1;h)=e^{x\cdot \zeta_1/h} 
(e^{\Phi_1(x,\gamma_1+i\gamma_2;h)}+r_1(x,\zeta_1; h))
\]
of the equation $(L_{\tilde A_1, \tilde q_1}-k^2) \tilde u_1=0$ in $B$, 
where $\Phi_1\in C^{\infty}(\overline{B})$.
By Remark \ref{rem_CGO_Lip}, $\Phi_1 \to \Phi_1^0$ in 
the $L^\infty$-norm as $h\to 0$, where 
$\Phi_1^0$ solves the equation
\begin{align}
\label{eq_Phi10}
(\gamma_1+i\gamma_2) \cdot \nabla \Phi_1^0 = -i(\gamma_1+i\gamma_2) \cdot \tilde A_1 \quad\text{in}\quad B.
\end{align}
To obtain a function that is zero on the plane $x_3=0$, we set
\begin{align}
\label{eq_u_1-cgo}
u_1(x) := \tilde u_1(x)-\tilde u_1(\tilde x),\quad x\in \Bm \cup l. 
\end{align}
Then it is easy to check that  $u_1\in W_1(\Bm)$.

We can similarly pick by Proposition \ref{prop_CGO_Lip} and Remark \ref{rem_com_geom_1},
complex geometric optics solutions $\tilde u_2$ in $H^2(B)$,
\[
\tilde u_2(x,\zeta_2;h)=e^{x\cdot \zeta_2/h} 
(e^{\Phi_2(x,-\gamma_1+i\gamma_2;h)}+r_2(x,\zeta_2; h))
\]
of the equation $(L_{\tilde A_2,\overline{\tilde q_2}}-k^2)\tilde u_2=0$ in 
$B$, 
where $\Phi_2\in C^{\infty}(\overline{B})$.
By Remark \ref{rem_CGO_Lip}, $\Phi_2 \to \Phi_2^0$ in 
the $L^\infty$-norm as $h\to 0$, where 
$\Phi_2^0$ solves the equation
\begin{align}
\label{eq_Phi20}
(-\gamma_1+i\gamma_2) \cdot \nabla \Phi_2^0 = -i(-\gamma_1+i\gamma_2) \cdot \tilde A_1 \quad\text{in}\quad B.
\end{align}
To obtain a function that is zero on the plane $x_3=0$, we set
\begin{align}
\label{eq_u_2-cgo}
u_2(x) := \tilde u_2(x)-\tilde u_2(\tilde x),\quad x\in \Bm \cup l. 
\end{align}
Then it is easy to check that  $u_1\in W_2^*(\Bm)$.

The next step is to substitute the complex geometric optics solutions $u_1$ and 
$u_2$, given by \eqref{eq_u_1-cgo} and \eqref{eq_u_2-cgo}, respectively,  into 
the integral identity \eqref{eq_sec3_4}. This  will be done in the Lemma bellow.
We will use the abbreviations $P_1(x) := e^{\Phi_1(x)}+r_1(x)$ and $P_2(x) :=
e^{\Phi_2(x)}+r_2(x)$, so that
\begin{align*}
u_1(x)&=e^{x\cdot\zeta_1/h}P_1(x) -e^{\tilde x\cdot\zeta_1/h}P_1(\tilde x), \\
u_2(x)&=e^{x\cdot\zeta_2/h}P_2(x) -e^{\tilde x\cdot\zeta_2/h}P_2(\tilde x).
\end{align*}
For future references, it 
will be convenient to compute the product of the phases that occur in the terms 
$u_1\ov{u}_2, \nabla u_1\ov{u_2}$ and $u_1 \nabla \ov{u_2}$
\begin{equation}
\label{eq_phases}
\begin{aligned}
e^{x\cdot\zeta_1/h}e^{x\cdot\overline{\zeta_2}/h}&=e^{ix\cdot\xi},\quad 
e^{\tilde x\cdot\zeta_1/h}e^{\tilde x\cdot\overline{\zeta_2}/h}=e^{i\tilde x 
\cdot\xi},\\
e^{\tilde x 
\cdot\zeta_1/h}e^{x\cdot\overline{\zeta_2}/h}
&=e^{ix\cdot\xi}e^{i(0,0,-2x_3)\cdot\zeta_1/h}
=e^{ix\cdot  \xi_--2\gamma_{1,3} x_3/h},\\
e^{x\cdot\zeta_1/h}e^{\tilde x\cdot\overline{\zeta_2}/h}
&=e^{i\tilde x\cdot\xi}e^{i(0,0,2x_3)\cdot\zeta_1/h} 
=e^{ix\cdot \xi_++2\gamma_{1,3} x_3/h},
\end{aligned}
\end{equation}
where $\gamma_j=(\gamma_{j,1},\gamma_{j,2},\gamma_{j,3})$, $j=1,2$ and 
\[
\xi_\pm=\bigg(\xi_1,\xi_2,\pm 
\frac{2}{h}\sqrt{1-\frac{h^2|\xi|^2}{4}}\gamma_{2,3}\bigg).
\]
We restrict the choices of $\gamma_1$, by assuming that
\begin{align} \label{eq_gamma-rest}
\gamma_{1,3}= 0 \quad \text{and} \quad  \gamma_{2,3}\ne 0.
\end{align}
We need these conditions for the proof of the next Lemma. 
The first condition makes the above phases
purely imaginary, which avoids exponential growth of the terms, as $h\to0$ . 
The second condition implies that $|\xi_\pm|\to \infty$ as $h\to 0$.
This will be needed since, we will use the Riemann-Lebesgue Lemma 
to eliminate unwanted imaginary exponentials. 

Finally it will also be convenient to explicitly state the following norm estimates, 
which follow from Proposition \ref{prop_CGO_Lip} 
\begin{equation}
\label{eq_rem_amp}
\begin{aligned}
&\|e^{\Phi_j}\|_{L^\infty}=\mathcal{O}(1),\quad \|\nabla 
e^{\Phi_j}\|_{L^\infty}=\mathcal{O}(h^{-1/3}),\\
&\|r_j\|_{L^2}=\mathcal{O}(h^{1/3}),\quad \|\nabla 
r_j\|_{L^2}=\mathcal{O}(h^{-2/3}),\quad j=1,2,
\end{aligned}
\end{equation}
as $h\to0$.

\begin{lem}
\label{inter_prop_32}
If the assumptions of Proposition \ref{identity} hold,
then 
\begin{equation}
\label{eq_identity_main_2}
(\gamma_1+i\gamma_2)\cdot\int_{B}(\tilde A_2-\tilde 
A_1)e^{ix\cdot\xi}e^{\Phi_1^0+\ov{\Phi_2^0}}dx=0,
\end{equation}
where $\gamma_1,\gamma_2$ and $\xi$ satisfy \eqref{eq_gamma-assum} and \eqref{eq_gamma-rest}.
\end{lem}
\begin{proof} We will prove the statement by
multiplying the integral equation of Proposition \ref{identity} by $h$,
when $u_1$ and $u_2$ are given by \eqref{eq_u_1-cgo} and \eqref{eq_u_2-cgo}, and
then take the limit as $h \to 0$.

We first show that for the second term in \eqref{eq_sec3_4} we have 
\begin{align}\label{eq_ua1}
h \int_{\Bm} (A_1^2-A_2^2 + q_1-q_2)u_1\ov{u}_2 \to 0,
\end{align}
as $h \to 0$.
Using the phase computations \eqref{eq_phases} we get that
\begin{align*}
u_1 \ov{u}_2 = 
&\quad e^{i\ksi \dotc x} P_1(x) \ov{P}_2(x) 
- e^{ix \cdot \xi_+} P_1(x) \ov{P}_2(\tilde x) \\
&- e^{ix \cdot \xi_-}P_1(\tilde x) \ov{P}_2(x) 
+ e^{i\xi \dotc \tilde x} P_1(\tilde x) \ov{P}_2(\tilde x).
\end{align*}
This is multiplied by an $L^\infty$ function in \eqref{eq_ua1}. Since we
restricted the choice of $\gamma_1$ to make the exponents purely imaginary, we
see easily using the estimates \eqref{eq_rem_amp} that \eqref{eq_ua1} holds.

Equation \eqref{eq_sec3_4} multiplied by $h$, is thus reduced in the limit to
\begin{align}
\label{eq_ua2}
\lim_{h \to 0} \bigg(
h \int_{\Bm} i(A_2-A_1)\dotc\nabla u_1 \ov{u}_2 
- h \int_{\Bm} i(A_2-A_1)\dotc u_1 \nabla \ov{u}_2 \bigg)=0.
\end{align}
We will proceed by examining the first term.
Using \eqref{eq_phases} we  write $\nabla u_1 \ov{u}_2$ as
\begin{align*}
\nabla u_1 \ov{u}_2 = 
&\frac{\zeta_1}{h} \big( 
e^{ix \cdot \xi} P_1(x) \ov{P_2(x)}-  e^{ix \cdot \xi_+} P_1(x) \ov{P_2(\tilde x)} \big)  \\
&+e^{ix \cdot \xi} \nabla P_1(x) \ov{P_2(x)}-  e^{ix \cdot \xi_+} \nabla P_1(x) \ov{P_2(\tilde x)} \\
&-\frac{\tilde \zeta_1}{h} \big( 
e^{ix \cdot \xi_-}  P_1(\tilde x) \ov{P_2(x)} - e^{i\tilde x \cdot \xi} P_1(\tilde x) 
\ov{P_2(\tilde x)} \big)  \\
&-e^{ix \cdot \xi_-} \nabla P_1(\tilde x) \ov{P_2(x)} + e^{i\tilde x \cdot \xi} \nabla P_1(x) 
\ov{P_2(\tilde x)},
\end{align*}
where $\tilde \zeta_j := \zeta_j \cdot (0,0,-1)$, $j=1,2$.
The terms of the product that do not contain the factor $1/h$, result in
integrals similar
to the one in \eqref{eq_ua1}. And one sees similarly using estimates  \eqref{eq_rem_amp} that they are 
zero in the limit of \eqref{eq_ua2}. 
The first term inside the limit in \eqref{eq_ua2} is therefore reduced to 
\begin{align*}
\lim_{h \to 0} 
\int_{\Bm} i(A_2-A_1)\dotc \big( 
&\zeta_1 e^{ix \cdot \xi} P_1(x) \ov{P_2(x)}   - \tilde \zeta_1 e^{ix \cdot \xi_-} P_1(\tilde x) \ov{P_2(x)} \\
-&\zeta_1  e^{ix \cdot \xi_+} P_1(x)\ov{P_2(\tilde x)}  - \tilde \zeta_1 e^{i\tilde x \cdot \xi} P_1(\tilde x) \ov{P_2(\tilde x)}  \big).
\end{align*}
Now we use the Riemann-Lebesgue Lemma to conclude that the terms with exponents
containing  $\xi_+$ and $\xi_-$ are zero in the limit. 
To see this, notice that by Remark \ref{rem_CGO_Lip}, we see
that $\|\Phi_i\|_{L^\infty(\Bm)} < C$, for some $C>0$, when $h$ is small enough. Estimates
\eqref{eq_rem_amp} show that
$\|r_i\|_{L^1(\Bm)} = \mathcal{O}(h^{1/3})$. Hence
$\|P_i\|_{L_1(\Bm)}<C$, for some $C>0$ when $h$ is small enough.
Finally we have $\xi_\pm \to \infty$, as $h\to0$, because of
the restrictions \eqref{eq_gamma-rest}.

The first term in \eqref{eq_ua2} is therefore 
\begin{align*}
\lim_{h \to 0} 
\int_{\Bm} i(A_2-A_1)\dotc \big( 
\zeta_1 e^{ix \cdot \xi} P_1(x) \ov{P_2(x)} 
+ \tilde \zeta_1 e^{i\tilde x \cdot \xi} P_1(\tilde x) \ov{P_2(\tilde x)}  \big)
\end{align*}
as $h\to 0$. The terms containing $r_i$ in the products of $P_1$ and $P_2$
are, because of \eqref{eq_rem_amp}, zero in the limit. 
The above limit is thus equal to
\begin{align*}
\lim_{h \to 0} 
\int_{\Bm} i(A_2-A_1)\dotc \big( 
\zeta_1 e^{ix \cdot \xi} 
e^{\Phi_1(x) + \ov{\Phi_2(x)}}  
+ \tilde \zeta_1 e^{i\tilde x \cdot \xi} 
e^{\Phi_1(\tilde x) + \ov{\Phi_2(\tilde x)}}  \big).
\end{align*}
Finally we split the integral and  do a change of variable in the second term
and  arrive at the expression
\begin{align}
\label{eq_ua3}
\lim_{h \to 0} \quad 
\int_{B} i(\tilde A_2- \tilde A_1)  \cdot \zeta_1 
e^{ix \cdot \xi} 
e^{\Phi_1(x) + \ov{\Phi_2(x)}},  
\end{align}
for the first term of \eqref{eq_ua2}.

Returning to the second term in \eqref{eq_ua2}, containing $u_1 \nabla
\ov{u_2}$. This is of the
same form as the first one. By doing the above derivation by simply exchanging the
roles of  $u_1$ and $\ov{u_2}$, we similarly see that the second term becomes
\begin{align}
\label{eq_ua4}
\lim_{h \to 0} \quad 
-\int_{B} i(\tilde A_2- \tilde A_1)
\cdot \ov{\zeta_2} 
e^{ix \cdot \xi} 
e^{\Phi_1(x) + \ov{\Phi_2(x)}}.  
\end{align}
Now $\zeta_1 \to (\gamma_1 + i \gamma_2)$ and $\ov{\zeta_2} \to - (\gamma_1+i\gamma_2)$,
as $h\to0$.
Thus by using \eqref{eq_ua3} with \eqref{eq_ua4}, we can rewrite \eqref{eq_ua2} as
\begin{align*}
\lim_{h \to 0} &\quad 
\int_{B} i(\tilde A_2-\tilde A_1)  \cdot \big( \zeta_1 
e^{ix \cdot \xi} 
e^{\Phi_1(x) + \ov{\Phi_2(x)}} 
-\ov{\zeta_2} 
e^{ix \cdot \xi} 
e^{\Phi_1(x) + \ov{\Phi_2(x)}} \big) \\
&=
\int_{B} i(\tilde A_2-\tilde A_1)  \cdot (\gamma_1 + i \gamma_2)
e^{ix \cdot \xi} 
e^{\Phi_1^0(x) + \ov{\Phi_2^0(x)}} = 0. 
\end{align*}
\end{proof}

The next Proposition shows that \eqref{eq_identity_main_2} holds even 
when the exponential function depending on $\Phi_i^0$, $i=1,2$ is removed.

\begin{prop}
\label{prop_32}
The equality \eqref{eq_identity_main_2} implies that 
\begin{equation}
\label{eq_identity_main_2_new}
(\gamma_1 + i \gamma_2) \cdot \int_{B} (\tilde A_2-\tilde 
A_1)e^{ix\cdot\xi}dx=0,
\end{equation}
for $\gamma_1,\gamma_2$ and $\xi$ which satisfy \eqref{eq_gamma-assum} and \eqref{eq_gamma-rest}.
\end{prop}

\begin{proof}
By \eqref{eq_Phi10} and \eqref{eq_Phi20} we have that
\begin{equation}
\label{eq_amplitude_sum}
(\gamma_1 + i\gamma_2)\cdot\nabla 
(\Phi_1^0+\overline{\Phi_2^0})=- i(\gamma_1+i\gamma_2)\cdot (\tilde 
A_1-\tilde A_2) \quad \textrm{in}\quad B. 
\end{equation}
Remark \ref{rem_g} furthermore implies that  the amplitude $e^{\Phi_1}$ in 
the definition of $u_1$ can be replaced  by $ge^{\Phi_1}$, 
if $g\in C^\infty(\overline{B})$ is a solution of 
\begin{equation}
\label{eq_g}
(\gamma_1+ i \gamma_2)\cdot \nabla g=0\quad  \textrm{in}\quad B.  
\end{equation}
Let $\Psi(x) :=\Phi_1^0(x)+\overline{\Phi_2^0}(x)$.
Then instead of \eqref{eq_identity_main_2} we can write, 
\[
(\gamma_1+ i\gamma_2)\cdot \int_{B} (\tilde A_2 -\tilde 
A_1)ge^{ix\cdot\xi}e^{\Psi(x)}dx=0.
\]
We conclude from \eqref{eq_amplitude_sum} that
\[
(\gamma_1 + i\gamma_2)\cdot (\tilde A_2 - \tilde A_1)ge^{\Psi} 
=-i (\gamma_1 + i\gamma_2 )\cdot(g\nabla e^{\Psi}),
\]
and therefore, we get
\begin{equation}
\label{eq_identity_main_3}
\int_{B} ge^{ix\cdot \xi}  (\gamma_1 + i\gamma_2)\cdot\nabla 
e^{\Psi} dx=0,
\end{equation}
for all $g$ satisfying \eqref{eq_g}. 

We pick a $\gamma_3$, with $|\gamma_3|=1$, so that we obtain an orthonormal basis
$\{\gamma_1,\gamma_2, \gamma_3\}$. Let $T$ be the coordinate transform into this
basis, i.e.
$y = Tx = (x\cdot\gamma_1, x\cdot\gamma_2, x\cdot\gamma_3)$. 
Set $z=y_1+i y_2$, so that $\p_{\bar z}=(\p_{y_1}+i\p_{y_2})/2$ and 
\[
(\gamma_1 + i\gamma_2)\cdot\nabla =2\p_{\bar z}. 
\]
Rewriting
\eqref{eq_identity_main_3} using this and a change of variable given by $T$ we have
\begin{align*}
\int_{TB} ge^{iy\cdot \xi}  \p_{\ov{z}} e^{\Psi} dy = 0,
\end{align*}
for all $g$ satisfying \eqref{eq_g}. 

Notice that $y\cdot \xi = y_3\xi_3$, since $\xi$ is in the y-coordinates of the
form $(0,0,\xi_3)$.
The above integral is therefore a Fourier transform w.r.t. $\xi_3$.
Let $g\in 
C^\infty(\overline{TB})$ satisfy $\p_{\bar z} g=0$ and 
be independent of $y_3$. Then taking the inverse Fourier transform we write
\begin{align*}
0 &= \int_{T_{y_3}} g \p_{\ov{z}} e^{\Psi} dy_1 dy_2 \\
&= \int_{T_{y_3}}  \p_{\ov{z}} (ge^{\Psi})dy_1 dy_2,
\end{align*}
where $T_{y_3 }:=TB\cap \Pi_{y_3}$ and 
$\Pi_{y_3}=\{(y_1,y_2,y_3):(y_1,y_2)\in\R^2\}$. 
Notice that the boundary of $T_{y_3}$ is piecewise smooth.
Multiplying the above by $2i$ and
using Stokes' theorem we get that
\begin{align}
0 &= 
2i \int_{T_{y_3 }} \p_{\ov{z}} (ge^{\Psi})dy_1 dy_2\nonumber  \\
&= \int_{T_{y_3 }}  \Curl (ge^{\Psi}, i ge^{\Psi},0)dy_1 dy_2\nonumber  \\
&= \int_{\p T_{y_3 }} (ge^{\Psi}, i ge^{\Psi},0) \cdot dl \nonumber \\
&= \int_{\p T_{y_3}} ge^{\Psi}dz, \label{eq_orth_hol_g}
\end{align}
for all holomorphic functions $g\in C^\infty(\overline{T_{y_3}})$.

Next we shall show that \eqref{eq_orth_hol_g} implies that there exists a 
nowhere vanishing holomorphic function $F\in C(\overline{T_{y_3}})$ such that
\begin{equation}
\label{eq_log}
F|_{\p T_{y_3}}=e^{\Psi}|_{\p T_{y_3}}. 
\end{equation}
 
To this end, we define $F$ to be   
\begin{align*}
F(z)=\frac{1}{2\pi i}\int_{\p T_{y_3}} 
\frac{e^{\Psi(\zeta)}}{\zeta-z}d\zeta,  \quad z\in\C\setminus\p T_{y_3}. 
\end{align*}
The function $F$ is holomorphic away from $\p T_{y_3}$.  As $e^{\Psi}$ is Lipschitz, 
we know because of the Plemelj-Sokhotski-Privalov formula (see e.g. \cite{Kress1}), that  
\begin{equation}
\label{eq_PSP_formula}
\lim_{z\to z_0,z\in T_{y_3}} F(z)-\lim_{z\to z_0,z\notin 
T_{y_3}}F(z)=e^{\Psi(z_0)},\quad z_0\in \p 
T_{y_3}.
\end{equation} 
Now the function $\zeta\mapsto (\zeta-z)^{-1}$ is holomorphic on $T_{y_3}$ when 
$z\notin T_{y_3}$. By choosing $g(z)=\zeta\mapsto (\zeta-z)^{-1}$ in \eqref{eq_orth_hol_g},
get therefore that $F(z)=0$, 
when $z\notin T_{y_3}$.  Hence, the second limit in 
\eqref{eq_PSP_formula} vanishes, and therefore, $F$ is holomorphic function on 
$T_{y_3}$, such that \eqref{eq_log} holds. 

Next we show that $F$ is non-vanishing in $T_{y_3}$.  When doing so, let  
$\p T_{y_3}$ be parametrized by $z=\gamma(t)$, and $N$ be the number of zeros 
of $F$ in $T_{y_3}$. Then by the argument principle, we get
\[
N=\frac{1}{2\pi i}\int_{\gamma}\frac{F'(z)}{F(z)}dz=\frac{1}{2\pi 
i}\int_{F \circ \gamma}\frac{1}{\zeta}d\zeta=\frac{1}{2\pi 
i}\int_{e^{\Psi \circ \gamma}} \frac{1}{\zeta}d\zeta= 0.
\]
To see that the  last integral is zero, notice that this the winding number of the 
path $e^{\Psi \circ \gamma}$. 
And that $e^{\Psi(\gamma(t))}$ is homotopic to 
the constant contour $\{1\}$, with the homotopy given by 
$e^{s\Psi(\gamma(t))}$, $s\in [0,1]$. 

Next, since $F$ is a non-vanishing holomorphic function on $T_{y_3}$ and 
$T_{y_3}$ is simply connected, it admits a holomorphic logarithm. Hence,  
\eqref{eq_log} implies that 
\[
(\log F)|_{\p T_{y_3}}=\Psi|_{\p T_{y_3}}.
\]
Because $\log F=\Psi$ is continuous on $\p T_{y_3}$, we have by the Cauchy theorem, 
\[
\int_{\p T_{y_3}} g \Psi dz=\int_{\p T_{y_3}} g 
\log F dz=0,
\]
where $g\in C^\infty(\overline{T_{y_3}})$ is  an arbitrary function such that 
$\p_{\bar z}g=0$.  
Using Stokes' formula as in \eqref{eq_orth_hol_g} allows us to write this as 
\[
\int_{T_{y_3}} g\p_{\bar z}\Psi dy_1 dy_2=0. 
\]
Taking the Fourier transform with respect to $y_3$, we get
\[
\int_{T(B)} e^{iy\cdot \xi} g\p_{\bar z}\Psi dy=0,
\]
for all $\xi=(0,0,\xi_3)$, $\xi_3\in\R$.   Hence, returning back to the 
$x$ variable, we obtain that 
\[
(\gamma_1 +i\gamma_2)\cdot \int_{B} e^{i x\cdot  \xi}g(x) \nabla \Psi(x) dx=0,
\]
where $g\in C^\infty(\overline{B})$ is such that $(\gamma_1 + 
i\gamma_2)\cdot \nabla g=0$ in $B$. 

Using \eqref{eq_amplitude_sum},  we finally get
\begin{equation}
\label{eq_identity_main_2_new_zero}
(\gamma_1 + i\gamma_2)\cdot \int_{B} (\tilde A_2-\tilde 
A_1)g(x)e^{ix\cdot\xi}dx=0. 
\end{equation}
Setting $g=1$, we obtain \eqref{eq_identity_main_2_new}.  

\end{proof}

By replacing the vector $\gamma_2$ by $-\gamma_2$ in 
\eqref{eq_identity_main_2_new}, we see that 
\begin{equation}
\label{eq_plane_2}
(\gamma_1 -i\gamma_2)\cdot \int_{B} (\tilde A_2 -\tilde 
A_1)e^{ix\cdot\xi}dx=0. 
\end{equation}
Hence,  \eqref{eq_identity_main_2_new} and \eqref{eq_plane_2} imply that
\begin{equation}
\label{eq_plane_3}
\gamma\cdot \int_{B} (\tilde A_2 -\tilde A_1)e^{ix\cdot\xi}dx=0,
\end{equation}
for all $\gamma\in \textrm{span}\{\gamma_1,\gamma_2\}$ and all $\xi\in \R^{3}$ 
such that 
\eqref{eq_gamma-assum} and  \eqref{eq_gamma-rest} hold.

In the proof of the next Proposition  we see that
\eqref{eq_identity_main_2_new} is actually a condition for having 
$\Curl (\tilde A_1 - \tilde A_2) = 0$. This is therefore the 
last step in proving that the DN-map determines the curl of the magnetic
potential. 

\begin{prop}
\label{prop_curl_thm_2}
Assume that $A_j,q_j$ and $\Gamma_j$, $j=1,2$ 
are as in Theorem \ref{thm_2_inverse} and that the DN-maps satisfy
\begin{align*}
\Lambda_{A_1,q_1}(f)|_{\Gamma_1}=\Lambda_{A_2,q_2}(f)|_{\Gamma_1},
\end{align*}
for any $f\in H^{3/2}_{\emph{\textrm{comp}}}(\p \HS)$, $\supp(f)\subset \Gamma_2$.
Then 
\begin{equation}
\label{eq_curl_2}
\nabla \times \tilde A_1= \nabla \times \tilde A_2\quad\textrm{in}\quad 
B. 
\end{equation}
\end{prop}

\begin{proof}

Assume that  $\xi\in\R^3$ is not on the line $L:=(0,0,t)$, $t \in \R$. Then the 
vectors $\gamma_1$ and $\gamma_2$ given by
\begin{alignat}{2}
\label{eq_vectors_special}
\tilde \gamma_1 &:= (-\xi_2, \xi_1,0) , \quad
&\gamma_1:=\tilde \gamma_1/|\tilde \gamma_1|, \nonumber \\
\tilde \gamma_2 &:= \xi \times \gamma_1, \quad
&\gamma_2:=\tilde \gamma_2/|\tilde \gamma_2|,
\end{alignat}
where $\xi \times \gamma_1$ stands for the vector cross product,
satisfy \eqref{eq_gamma-rest} and \eqref{eq_gamma-assum}.
Thus, for any vector $\xi\in\R^3 \setminus L$,
\eqref{eq_plane_3} says that
\begin{equation}
\label{eq_plane_4}
\gamma \cdot  v(\xi)=0, \quad v(\xi):=\widehat{\tilde 
A_2\chi}(\xi)-\widehat{\tilde A_1\chi}(\xi),
\end{equation}
for all $\gamma\in \textrm{span}\{\gamma_1,\gamma_2\}$. Here $\chi$ is the 
characteristic function of the set $B$.  
For any vector $\xi\in \R^3$, we have the following decomposition,
\[
v(\xi)=v_{\xi}(\xi)+v_\perp(\xi),
\] 
where $\Re v_{\xi}(\xi)$, $\Im v_{\xi}(\xi)$ are multiples of $\xi$, and  $\Re 
v_\perp(\xi)$, $\Im v_\perp(\xi)$ are orthogonal to $\xi$.  Now we have $\Re 
v_\perp(\xi), \Im v_\perp(\xi)\in \textrm{span}\{\gamma_1,\gamma_2\}$, and 
therefore, it follows from \eqref{eq_plane_4} 
that $v_\perp(\xi)=0$, 
for all $\xi\in \R^3\setminus L$. 

Hence, 
$v(\xi)=\alpha(\xi)\xi$,
so that that 
\[
\xi\times v(\xi)=0,
\]
for all $\xi\in \R^3\setminus L$, and thus, everywhere, 
by the analyticity of the Fourier transform.  Taking the inverse Fourier 
transform, we obtain \eqref{eq_curl_2}.
\end{proof}

\subsection{Determining the electric potential}

In order to complete the proof  of Theorem \ref{thm_2_inverse}, we need 
to show that the electric potential is also determined by the
DN-map. Again we assume that 
$A_j,q_j$ and $\Gamma_j$, $j=1,2$ 
are as in Theorem \ref{thm_2_inverse} and that the DN-maps satisfy
\eqref{eq_data_inv}, and hence  \eqref{eq_data_inv_1}.

Since $B$ is simply connected,
it follows from the Helmholtz decomposition of $\tilde A_1-\tilde A_2$  and \eqref{eq_curl_2} that 
there exists $\psi\in C^{1,1}(\ov{B})$ with $\psi=0$ near $\p B$ 
such that 
\[
\tilde A_1=\tilde A_2+\nabla \psi\quad\textrm{in}\quad B. 
\]
We extend $\psi$ to a function of class $C^{1,1}$ on all of $\R^3$ such that 
$\psi=0$ on $\R^3\setminus\ov{B}$. Then 
\[
\tilde A_1=\tilde A_2+\nabla \psi\quad\textrm{in}\quad \R^3. 
\]
In particular, $\psi=0$ on $\tilde \Gamma_1\cup\tilde \Gamma_2$. 
It follows then
from Lemma \ref{gauge_inv} part (\textit{i}) and \eqref{eq_data_inv_1}
that for all $f$ with 
$\supp(f)\subset\tilde \Gamma_2$,
\[
\Lambda_{A_1,q_1}(f)|_{\tilde \Gamma_1}=
\Lambda_{A_2,q_2}(f)|_{\tilde \Gamma_1}=\Lambda_{A_2+\nabla 
\psi,q_2}(f)|_{\tilde \Gamma_1}=\Lambda_{A_1,q_2}(f)|_{\tilde \Gamma_1}.
\]
We can now, by Remark \ref{rem_ie} use this with Proposition \ref{identity}.
That is we consider equation \eqref{eq_sec3_4}, in the case $A_1=A_2$. This gives
\begin{equation}
\label{eq_recovering_q}
\int_{\Bm}(q_1-q_2)u_1\ov{u_2}dx=0,
\end{equation}
for all $u_1\in W_1(\Bm)$ and $u_2\in W_2^*(\Bm)$. 

Choosing in \eqref{eq_recovering_q} $u_1$ and $u_2$ as the complex geometric 
optics solutions, given by \eqref{eq_u_1-cgo} and \eqref{eq_u_2-cgo}, and 
letting $h\to 0$, we have
\begin{equation}
\label{eq_recovering_q_2}
\int_{B}(\tilde q_1-\tilde 
q_2)e^{ix\cdot\xi}e^{\Phi_1^0(x)+\overline{\Phi_2^0 (x)}}dx=0.
\end{equation}
By Remark \ref{rem_g}
$e^{\Phi_1}$ in the definition \eqref{eq_u_1-cgo} of $u_1$ can be replaced  by 
$ge^{\Phi_1}$ if $g\in C^\infty(\overline{B})$ is a solution of 
\[
(\gamma_1+ i \gamma_2)\cdot \nabla g=0\quad  \textrm{in}\quad B.  
\]
Then \eqref{eq_recovering_q_2} can be replaced by 
\[
\int_{B}(\tilde q_1-\tilde 
q_2)g(x)e^{ix\cdot\xi}e^{\Phi_1^0(x)+\overline{\Phi_2^0 (x)}}dx=0.
\]
Now \eqref{eq_amplitude_sum} has the form,
\[
(\gamma_1 + i\gamma_2)\cdot\nabla (\Phi_1^0+\overline{\Phi_2^0})=0\quad 
\textrm{in}\quad B, 
\]
since we consider that $\tilde A_1 = \tilde A_2$.
Thus, we can take $g=e^{-(\Phi_1^0+\overline{\Phi_2^0)}}$ and  obtain that 
\begin{equation}
\label{eq_recovering_q_3}
\int_{B}(\tilde q_1-\tilde q_2)e^{ix\cdot\xi}dx=0,
\end{equation}
for all $\xi\in \R^3$ such that there exist $\gamma_1,\gamma_2\in \R^3$, 
satisfying \eqref{eq_gamma-assum} and  \eqref{eq_gamma-rest}.
Since for any $\xi\in\R^3$ not of the form $\xi=(0,0,\xi_3)$, the vectors, 
given by \eqref{eq_vectors_special}, satisfy \eqref{eq_gamma-assum} and  \eqref{eq_gamma-rest},
we conclude 
that \eqref{eq_recovering_q_3} holds for all  $\xi\in\R^3$ 
except those of the form $\xi=(0,0,\xi_3)$, and therefore, by analyticity of the Fourier 
transform, for all $\xi\in \R^3$. Hence,  $q_1=q_2$ in $\Bm$. 
This completes the proof of Theorem \ref{thm_2_inverse}.

\newpage
\section{Appendix}

\subsection{Magnetic Green's formulas}
\label{sec:EGF}

Let us first recall, following \cite{DosSantos1},  the standard Green formula 
applied to the magnetic Schr\"odinger operator.  

\begin{lem} \label{MagGFI} 
Suppose that $\Omega \subset \R^3$ is open and bounded, with piecewise $C^{\,1}$
boundary. Let $A \in W^{1,\infty}(\Omega,\R^3)$ and  $q \in 
L^{\infty}(\Omega)$. Then we have, 
\begin{align*}
& \quad (\Ld u, v)_{L^2(\Omega)} - (u , L_{A,\ov{q}} v)_{L^2(\Omega)} \\
&= (u , (\partial_n + iA\cdot n)
v)_{L^2(\partial \Omega)}-  ((\partial_n + iA\cdot n) u, v)_{L^2(\partial 
\Omega)},
\end{align*}
for all $u,v \in H^1(\Omega)$, with $\Delta u, \Delta v \in L^2(\Omega)$,
where $n$ is the exterior unit normal to $\p \Omega$.
\end{lem}

We shall also need a version of the above result where $\Omega$ 
is replaced by $\HS$. We shall then need to put some restrictions on $v$ and 
$u$, 
because $\HS$ is unbounded. To this end we assume that $u$ and $v$ are 
solutions to
the Helmholtz equation outside some compact set, 
that obey some form of radiation condition. To be 
precise, let $A \in \Wc^{1,\infty}(\HS,\R^3)$,  $q \in
\Lc^{\infty}(\HS)$, and let 
$u\in H^2_{\textrm{loc}}(\ov{\HS})$ be such that 
\[
(L_{A,q}-k^2) u=0\quad \textrm{in}\quad \HS,
\]
$\supp(u|_{\p\HS})$ is compact, and $u$ is outgoing.  Assume also that $v\in 
H^2_{\textrm{loc}}(\ov{\HS})$ satisfies 
\[
(L_{A,\ov{q}}-k^2)v\in L^2_{\textrm{comp}}(\CHS), 
\]
$\supp(v|_{\p \HS})$ is compact, and $v$ is incoming.

\begin{lem} \label{MagGFII} 
With $u$ and $v$ as above, we have 
\begin{equation}
\label{eq:MagGFIIeq}
\begin{aligned}
& \quad ((\Ld-k^2) u, v)_{L^2(\HS)} - (u , (L_{A,\ov{q}}-k^2) 
v)_{L^2(\HS)} \\
&= (u , (\partial_n + iA\cdot n)
v)_{L^2(\p\HS)}-  ((\partial_n + iA\cdot n) u, v)_{L^2(\p\HS)}.
\end{aligned}
\end{equation}
\end{lem}

\begin{proof}
Let $B_R := \{x\in \R^3 \,\big|\, |x|<R\}$ be an open ball in $\R^3$ of radius $R$,  and 
choose $R>0$ large enough so that 
\[
\supp(A),\supp(q)\subset B_R. 
\]
Set $\Omega = \HS \cap B_R$.  By Lemma \ref{MagGFI},  we know that
\begin{align*}
& \quad ((\Ld-k^2) u, v)_{L^2(\Omega)} - (u , (L_{A,\ov{q}}-k^2) 
v)_{L^2(\Omega)} \\
&= (u , (\partial_n + iA\cdot n)
v)_{L^2(\partial \Omega)}-  ((\partial_n + iA\cdot n) u, v)_{L^2(\partial 
\Omega)}. 
\end{align*}
Thus, to obtain \eqref{eq:MagGFIIeq} we need to show that 
\begin{equation}
\label{eq_green_pr1}
\int_{\p B_R\cap \HS} (u  \ov{\partial_n  v}  - (\partial_n u)\ov{v} ) d S_R 
\to 0, \quad R\to \infty.
\end{equation}
Let us rewrite the left hand side of the above as follows, 
\begin{align*}
 \int_{\p B_R\cap \HS}(\partial_n \ov{v} -ik\ov{v})udS_R - \int_{\p B_R\cap 
\HS}(\partial_n u -iku)\ov{v}dS_R. 
\end{align*}
We show that first term goes to zero as $R \to \infty$. The second term can be
handled in the same way. Applying Cauchy-Schwarz gives
\begin{align*}
\bigg|\int_{\p B_R\cap \HS}(\partial_n \ov{v} -ik\ov{v})u dS_R\bigg|^2
\leq \int_{\p B_R\cap \HS} |\partial_n \ov{v} -ik\ov{v}|^2dS_R \int_{\p B_R\cap 
\HS} | u |^2 dS_R.
\end{align*}
Here the first integral goes to zero, since 
$\ov{\partial_n \ov{v} -ik\ov{v}} = \partial_n v +ikv$ and $|\partial_n v
+ikv|^2$ is $o(1/r^2)$ as $r=|x|\to \infty$, since $v$ is incoming. The second 
integral is bounded, because of Lemma
\ref{SRChlp}. We conclude that \eqref{eq_green_pr1} holds.  
\end{proof}

\subsection{Carleman estimates and solvability}
\label{sec:CE}

Let $\Omega\subset\R^n$, $n\ge 2$, be a bounded domain with $C^\infty$ 
boundary, and let $\varphi(x) = \alpha \cdot x$  with $\alpha \in \R^n$, 
$|\alpha|=1$. Consider the conjugated operator
\begin{align*}
L_\varphi :=  e^{\varphi / h} h^2 \Ld e^{-\varphi/h}.
\end{align*}
Note that $L_\varphi$ depends on $A,q$ and $h$. In the beginning  of this 
subsection we shall establish the following Carleman estimate, where we write
\begin{align*}
\| u \|_{H_{\textrm{scl}}^1(\Omega)}^2 = \|u\|_{L^2(\Omega)}^2 + \|h \nabla u
\|^2_{L^2(\Omega)}.
\end{align*}

\begin{thm}
\label{carlI} 
Let $A \in W^{1,\infty}(\Omega,\C^n)$ and $q \in
L^{\infty}(\Omega,\C)$. Then there exist   $C > 0$ and   $h_0 > 0 $ such that 
for all
$u \in C^{\infty}_0(\Omega)$, we have 
\begin{align}
h \| u \|_{H_{\emph{scl}}^1(\Omega)} \leq 
C \| L_{\varphi} u \|_{L^2(\Omega)},
\end{align}
when $0 < h \leq h_0$.
\end{thm}

\begin{proof} In what follows, the $L^2(\Omega)$-norm is abbreviated to 
$\|\cdot\|$. 
Let $\epsilon > 0$. Define
\begin{align*}
\varphi_{\epsilon} (x)  = \alpha \dotc x + \frac{\epsilon}{2}(\alpha \dotc x)^2.
\end{align*}
Denote by $L_{0,\varphi_\epsilon} 
:= -e^{\varphi_\epsilon/h}h^2\Delta e^{-\varphi_\epsilon/h}$.
This can be decomposed as $L_{0,\varphi_\epsilon} = A_\epsilon + i
B_\epsilon$, where
\begin{align*}
A_\epsilon & :=   -h^2\Delta - (1+\epsilon\alpha\dotc x)^2, \\
B_\epsilon & := -   2 h(1+\epsilon\alpha\dotc x) \alpha \dotc i \nabla - i 
\epsilon h.
\end{align*}
A direct calculation gives that
\begin{align}
\| L_{0,\varphi_\epsilon} u \|^2 =
\| A_\epsilon u\|^2  
+ \| B_\epsilon u\|^2  + i ([A_\epsilon , B_\epsilon]u , u),\quad u\in 
C^\infty_0(\Omega). \label{eq:commt}
\end{align}
Another straightforward calculation shows that the commutator can be written  
as  
\begin{align*}
i[A_\epsilon , B_\epsilon] = -4 \epsilon h^3 
\sum_{j,k = 1}^3 \alpha_j\alpha_k \partial_j \partial_k
+ 4 h \epsilon (1 + \epsilon\alpha\dotc x)^2.
\end{align*}
The first term is an operator with a positive semi-definite semiclassical 
symbol of the form 
$4\epsilon h(\alpha\cdot\xi)^2\ge 0$. 
The inner product in \eqref{eq:commt}, with this part is therefore
non-negative.
We can hence drop it and estimate \eqref{eq:commt} as follows, 
\begin{align}
\| L_{0,\varphi_\epsilon} u \|^2 
&\geq  \| A_\epsilon u\|^2  + 
 4 h\epsilon \| (1 + \epsilon \alpha \dotc x) u \|^2 \nonumber \\
&\geq  \| A_\epsilon u\|^2  + 
 h\epsilon \| u \|^2, \label{eq:cestep1}
\end{align}
for all $\epsilon >0$ such that  $\epsilon |\alpha \dotc x| \leq 1/2$, $x\in 
\Omega$. 

The next step is to obtain a similar estimate for the first order perturbation
$L_{\varphi_\epsilon} :=  e^{\varphi_\epsilon/h}h^2\Ld e^{-\varphi_\epsilon/h}$.
We decompose this as  $L_{\varphi_\epsilon} =  L_{0,\varphi_\epsilon} +
Q_{\epsilon}$, where
\begin{align*}
Q_\epsilon & := e^{\varphi_\epsilon/h} h^2(-2iA \dotc \nabla -i\nabla \dotc A + 
A^2 +q )e^{-\varphi_\epsilon/h}.
\end{align*}
This can be estimated as 
\begin{align*}
\| Q_\epsilon u \|^2 &= h^4\| -2iA \dotc \nabla u + 2iA \dotc \nabla
\frac{\varphi_\epsilon}{h} u -i\nabla \dotc Au + A^2u +qu\|^2 \\
&\leq C \big( h^4\| \nabla u \|^2 +  h^2\| u \|^2 \big),
\end{align*}
when $h$ is sufficiently small. For  $L_{\varphi_\epsilon}$, we get
\begin{align}
\| L_{0,\varphi_\epsilon} u \|^2 
&\leq C \big( \| L_{\varphi_\epsilon} u \|^2 +  h^4\| \nabla u \|^2 +
h^2\| u \|^2 \big),
\label{eq:ce2}
\end{align}
when $h$ is sufficiently small. Next  we rewrite the gradient term and use the
Cauchy inequality to obtain
\begin{align}
h^2\| \nabla u\|^2 &= (-h^2\Delta u, u)  \nonumber \\  
&= (A_\epsilon u, u) + \| (1+\epsilon \alpha \dotc x) u \|^2 \nonumber \\
&\leq \frac{\delta}{2} \| A_{\epsilon} u \|^2 + \frac{2}{\delta} \| u\|^2 
+ 4 \| u \|^2, 
\label{eq:cegrad}
\end{align}
where $\delta > 0$.
Combining this with \eqref{eq:cestep1} and \eqref{eq:ce2}, gives
\begin{align*}
\| A_\epsilon u\|^2  +  h\epsilon \| u \|^2
&\leq 
C \bigg( \| L_{\varphi_\epsilon} u \|^2 + h^2\| u \|^2  \\ 
&+ \frac{\delta}{2} h^2\| A_{\epsilon} u \|^2 + (\frac{2}{\delta} + 4)h^2 \| u 
\|^2 \bigg ). 
\end{align*}
Choose $\delta = 2/C$. Rearranging the above gives
\begin{align*}
(1-h^2)\| A_\epsilon u\|^2  +  h(\epsilon - (C^2+5C)h) \| u \|^2
&\leq 
C \| L_{\varphi_\epsilon} u \|^2.
\end{align*}
We may assume that $h < 1/2$. The first term is then larger than  $h^2\|
A_\epsilon u\|^2$. 
Next we pick $\epsilon$ so that $(\epsilon-(C^2+5C)h)=6h$, i.e. $\epsilon=Mh$, 
$M>0$ is fixed.  This gives
\begin{align*}
h^2\| A_\epsilon u\|^2  +  6h^2 \| u \|^2
&\leq 
C \| L_{\varphi_\epsilon} u \|^2.
\end{align*}
Using \eqref{eq:cegrad} with $\delta = 2$, gives then 
\begin{align}
h^4 \| \nabla u\|^2  +  h^2 \| u \|^2 &\leq  C \| L_{\varphi_\epsilon} u \|^2.
\label{eq:s1est}
\end{align}
Written in another way we have 
\begin{align*}
h^2\| u \|^2_{H^1_{\textrm{scl}}} &\leq  C \| L_{\varphi_\epsilon} u \|^2_{L^2},
\end{align*} 
which is almost the desired estimate.

To finish the proof we need to replace $L_{\varphi_\epsilon}$ in 
\eqref{eq:s1est} by
$L_\varphi$. To this end
let $g:= \epsilon(\alpha \dotc x)^2/2$ and let $u = e^{g/h}\tilde{u}$. Notice 
that $g/h=M(\alpha \dotc x)^2/2$ is independent of $h$. 
Estimate \eqref{eq:s1est} gives now 
\begin{align*}
h^4\|\nabla u\|^2  +  h^2\| u \|^2
\leq  C \| e^{g/h} L_{\varphi} \tilde{u} \|^2
\leq  C' \| L_{\varphi} \tilde{u} \|^2,
\end{align*} 
for some constants $C,C'>0$. To obtain \eqref{eq:s1est} with
$L_\varphi$, we need only to show that 
\begin{align}
h^4\|\nabla \tilde{u}\|^2 + h^2\|\tilde{u}\|^2 \leq C(h^4\|\nabla u\|^2 + h^2\| 
u \|^2)
\label{eq:semicl1}
\end{align}
for some constant $C>0$. Using the  triangle
and Cauchy inequalities we see that 
\begin{align*}
\| e^{g/h} \nabla \tilde{u}\|^2 
&\leq   2 \| \nabla (e^{g/h}) \tilde{u} + e^{g/h} \nabla \tilde{u}  \|^2 
+  2\|\nabla(g/h) e^{g/h} \tilde{u}\|^2 \\
&\leq   2 \| \nabla (e^{g/h} \tilde{u}) \|^2 +  C h^{-2}\| e^{g/h}
\tilde{u}\|^2,
\end{align*} 
for some some constant $C>0$. Hence, 
\begin{align*}
 h^4 \|\nabla \tilde{u}\|^2 
&\leq   C( h^4\|\nabla u \|^2 + h^2\| u \|^2),
\end{align*} 
for some constant $C > 0$, which shows that \eqref{eq:semicl1} holds. The proof 
is complete. 

\end{proof}

Let 
\begin{align*}
\| u \|_{H_{\textrm{scl}}^s(\R^n)}^2 =  (2\pi)^{-3}\int_{\R^n} (1+h^2|\xi|^2)^s 
|\hat{u}(\xi)|^2 d \xi,\quad s\in \R.
\end{align*}
We have the following consequence of Theorem \ref{carlI},  see also 
\cite{DosSantos1}. 

\begin{cor} \label{carlII}
Let $A \in W^{1,\infty}(\Omega,\C^n)$ and $q \in
L^{\infty}(\Omega,\C)$. Then there exist  $C > 0$ and  $h_0 > 0 $ such that for 
all
$u \in C^{\infty}_0(\Omega)$, we have 
\begin{align}
h \| u \|_{L^2(\Omega)}  \leq 
C \| L_{\varphi} u \|_{H_{\emph{scl}}^{-1}(\R^n)},
\end{align}
when $0 < h \leq h_0$.
\end{cor}

The formal adjoint of $L_\varphi$ is given by 
\begin{align*}
L_\varphi^* =  e^{-\varphi / h} h^2 L_{\ov{A},\overline{q}} e^{\varphi/h},
\end{align*}
and Corollary \ref{carlII} still holds for $L_\varphi^*$.

Next we prove the following solvability result, see also \cite[Proposition 
4.3]{KnuSalo}. 
\begin{prop}
Let $A \in W^{1,\infty}(\Omega,\C^n)$, $q \in
L^{\infty}(\Omega,\C)$, $\alpha \in \R^n$, $|\alpha|=1$ and $\varphi(x) =
\alpha \cdot x$. Then there is  $C > 0$ and  $h_0 > 0 $ such that for all $h\in 
(0,h_0]$, and any $f\in L^2(\Omega)$, the equation
\[
L_{\varphi} u=f\quad\textrm{in}\quad \Omega,
\]
has a solution $u\in H^1(\Omega)$ with 
\begin{align*}
\|u \|_{H^1_{\emph{\textrm{scl}}}(\Omega)}\le \frac{C}{h}\|f\|_{L^2(\Omega)}. 
\end{align*}
 \end{prop}

\begin{proof}
The Carleman estimate of Corollary \ref{carlII} applied to $L_\varphi^*$ shows 
that $L_\varphi^*$ is
injective on $C^\infty_0(\Omega)$.  This lets us define the functional
$T:L_\varphi^*C^\infty_0(\Omega) \nuoli \C$,
given  by 
\begin{align*}
T(w) := \big((L_\varphi^*)^{-1}w,f \big)_{L^2(\R^n)}.
\end{align*}
The set 
$\Dom(T) \subset L^{\infty}(\Omega) \subset H^{-1}_{\textrm{scl}}(\R^n)$ is a 
linear 
subspace. The linear functional $T$ is moreover bounded, since by Corollary 
\ref{carlII},
\begin{equation}
\label{eq_norm_T}
\begin{aligned}
|T(w)|&= |\big((L_\varphi^*)^{-1}w,f \big)_{L^2(\Omega)}|\\
&\leq
\| (L_\varphi^*)^{-1}w\|_{L^2(\Omega)} \| f\|_{L^2(\Omega)} \\
&\leq
\frac{C}{h} \| f\|_{L^2(\Omega)}\| w\|_{H^{-1}_{\textrm{scl}}(\R^n)},\quad C>0.
\end{aligned}
\end{equation}
The Hahn-Banach theorem allows us to extend $T$, without
increasing its norm, to an operator
$\tilde{T}: H^{-1}_{\textrm{scl}}(\R^n) \nuoli \C$. By the  Riesz 
representation theorem there
exists $r \in H^{-1}_{\textrm{scl}}(\R^n)$ such that 
\begin{align*}
\tilde{T}(w) = (w,r)_{H^{-1}_{\textrm{scl}}},\quad \text{for $w \in
H^{-1}_{\textrm{scl}}(\R^n)$},
\end{align*}
and 
\begin{equation}
\label{eq_norm_T_2}
\|r\|_{H^{-1}_{\textrm{scl}}(\R^n)}=\|\tilde T\|\le \|T\|\le \frac{C}{h} \| 
f\|_{L^2(\Omega)}.
\end{equation}
Here we have used \eqref{eq_norm_T}.

Furthermore, we have
\[
(w,r)_{H^{-1}_{\textrm{scl}}}=(w,u)_{L^2},
\]
with 
\[
u=\mathcal{F}^{-1}(1+h^2\xi^2)^{-1}\mathcal{F} r\in H^1_{\textrm{scl}}(\R^n). 
\]
Here $\mathcal{F} $ is the Fourier transformation on $\R^n$.

We now show that  $u$ solves $L_\varphi u=f$ in the weak sense in $\Omega$. 
For every $\psi \in C^\infty_0(\Omega)$ we have
\begin{align*}
(L_\varphi u,\psi)_{L^2(\Omega)} 
&= (u,L_\varphi^*\psi)_{L^2(\R^n)} \\
&= \ov{T(L_\varphi^*\psi)} \\
&= \ov{\big((L_\varphi^*)^{-1}L_\varphi^*\psi, f \big)_{L^2(\R^n)}} \\
&= (f,\psi)_{L^2(\Omega)}.
\end{align*}
Hence $u$ is a weak solution.

To obtain the norm estimate, we observe that 
$\|u\|_{H^1_{\textrm{scl}}(\R^n)}=\|r\|_{H^{-1}_{\textrm{scl}}(\R^n)}$ and use 
\eqref{eq_norm_T_2}.  The proof is complete. 

\end{proof}

\subsection{The unique continuation principle}

In this work we make heavy use of the so called \textit{unique
continuation principle}. The unique continuation principle can be seen as an 
extension of the  familiar property that an analytic function that 
is zero on some
open set is identically zero.

Let $\Omega\subset\R^n$ be an open connected set, and let
\[
Pu=\sum_{i,j=1}^n a_{ij}(x)\p_i\!\p_j\!u+\sum_{i} b_i(x)\p_i\!u+c(x)u. 
\]
Here $a_{ij}\in C^1(\ov{\Omega})$ are real-valued, $a_{ij}=a_{ji}$, and there 
is $C>0$ so that 
\[
\sum_{i,j=1}^n a_{ij}(x) \xi_i \xi_j \geq C|\xi|^2, \quad x \in 
\ov{\Omega},\quad \xi\in\R^n.
\]
Furthermore, 
$b_i\in L^\infty(\Omega,\C)$ and $c\in L^\infty(\Omega, \C)$. We have the 
following result, see  \cite{Choulli} and \cite{Leis1}.

\begin{thm}\label{UCP} 
Let $u\in H^2_{\emph{loc}}(\Omega)$ be such that $Pu=0$ in $\Omega$. Let 
$\omega\subset\Omega$ be open non-empty. If $u=0$ on $\omega$, then $u$ 
vanishes identically in $\Omega$. 

\end{thm}

\begin{cor} \label{UCP_boundary}  Assume that $\p \Omega$ is of class $C^2$. 
Let $\Gamma\subset \p \Omega$ be open non-empty. Let $u\in H^2(\Omega)$ be such 
that $Pu=0$ in $\Omega$. Assume that 
\[
u=\mathcal{B}_{\nu}u=0\quad \textrm{on}\quad \Gamma.
\]
Here $\mathcal{B}_{\nu}u$ is  the conormal derivative of $u$, given by
\[
\mathcal{B}_{\nu}u=\sum_{i, j=1}^n \nu_i (a_{ij}\p_{j} u)|_{\p \Omega}\in 
H^{1/2}(\p \Omega).
\]
Then $u$ vanishes identically in $\Omega$. 

\end{cor}

\subsection{Rellich's lemma}

Rellich's lemma is a fundamental result in the scattering theory of the 
Helmholtz
equation, see  see e.g. \cite{Colton1}.

\begin{prop}\label{RL} Let $k>0$ and let $u \in \mathcal{D}'(\R^3)$ satisfy the 
Helmholtz equation $(-\Delta-k^2)u=0$ outside a  ball $B$ in $\R^3$.  Assume 
that 
\begin{align*}
\lim_{R \to \infty} \int_{|x|=R} |u|^2 dS\to 0.
\end{align*}
Then $u \equiv 0$ in $\R^3 \setminus \ov{B}$.
\end{prop}

\newpage

\bibliographystyle{plain}
\bibliography{ref}

\end{document}